\documentclass[11pt]{amsart}
\usepackage[utf8]{inputenc}
\usepackage{fontenc}
\usepackage{amsfonts}
\usepackage{amssymb}
\usepackage{amsmath}
\usepackage{amsthm}\usepackage{mathtools}
\usepackage{enumerate}
\usepackage{hyperref}\usepackage{enumitem}
\usepackage{mathrsfs}
\usepackage{tikz}\usetikzlibrary{calc}
\usepackage{marginnote}
\usepackage{xcolor,enumitem}
\usepackage{soul}\usepackage{tikz}
\usepackage[all,cmtip]{xy}

\newcommand{\mathscripty}{\mathscr}

\newcommand{\e}{\varepsilon}
\newcommand{\eps}{\varepsilon}

\newcommand{\Z}{\mathbb{Z}}
\newcommand{\D}{\mathbb{D}}
\newcommand{\N}{\mathbb{N}}
\newcommand{\C}{\mathbb{C}}
\newcommand{\bA}{A}

\newcommand{\cV}{\mathcal{V}}
\newcommand{\cY}{\mathcal{Y}}
\newcommand{\cX}{\mathcal{X}}

\theoremstyle{Case1}

\theoremstyle{Case2}

\newcommand{\Lip}{\mathrm{Lip}}

\newcommand{\cstu}{\mathrm{C}^*_u}
\newcommand{\csts}{\mathrm{C}^*_s}
\newcommand{\roeq}{\mathrm{Q}^*_u}
\newtheorem*{rigprob*}{Rigidity Problem for Uniform Roe Algebras}
\newtheorem*{rigprobcorona*}{Rigidity Problem for Uniform Roe Coronas}

\newcommand{\SI}{\mathscripty{I}}
\newcommand{\SJ}{\mathscripty{J}}

\newcommand{\cst}{\mathrm{C}^*}
\newcommand{\cstar}{$\mathrm{C}^*$}

\newcommand{\cU}{\mathcal{U}}

\newcommand{\cP}{\mathcal{P}}
\newcommand{\bbN}{\mathbb{N}}
\newcommand{\cB}{\mathcal{B}}
\newcommand{\cK}{\mathcal{K}}

\newtheorem{theorem}{Theorem}[section]
\newtheorem*{theorem*}{Theorem}
\newtheorem{proposition}[theorem]{Proposition}
\newtheorem{problem}[theorem]{Problem}
\newtheorem*{proposition*}{Proposition}
\newtheorem{lemma}[theorem]{Lemma}
\newtheorem*{lemma*}{Lemma}
\newtheorem{corollary}[theorem]{Corollary}
\newtheorem*{corollary*}{Corollar}

\newtheorem*{fact*}{Fact}
\theoremstyle{definition}
\newtheorem{definition}[theorem]{Definition}
\newtheorem*{definition*}{Definition}
\newtheorem{claim}[theorem]{Claim}
\newtheorem*{claim*}{Claim}
\newtheorem*{proofsketch}{{\normalfont\textit{Sketch of the proof}}}

\newtheorem*{conjecture*}{Conjecture}

\newtheorem{assumption}[theorem]{Assumption}

\theoremstyle{remark}

\newtheorem*{example*}{Example}
\newtheorem{remark}[theorem]{Remark}
\newtheorem*{remark*}{Remark}

\newtheorem*{note*}{Note}
\newtheorem*{question*}{Question}

\newcommand{\norm}[1]{\left\lVert #1 \right\rVert}

\DeclareMathOperator{\supp}{supp}

\DeclareMathOperator{\propg}{prop}

\DeclareMathOperator{\rank}{rank}

\DeclareMathOperator{\diam}{diam}

\newcounter{my_enumerate_counter}
\newcommand{\pushcounter}{\setcounter{my_enumerate_counter}{\value{enumi}}}
\newcommand{\popcounter}{\setcounter{enumi}{\value{my_enumerate_counter}}}

\usepackage{enumitem}

\begin{document}

\title{Embeddings of uniform Roe algebras}%

\author[B. M. Braga]{Bruno M. Braga}
\address[B. M. Braga]{Department of Mathematics and Statistics,
York University,
4700 Keele Street,
Toronto, Ontario, Canada, M3J
1P3}
\email{demendoncabraga@gmail.com}
\urladdr{https://sites.google.com/site/demendoncabraga}

\author[I. Farah]{Ilijas Farah}
\address[I. Farah]{Department of Mathematics and Statistics,
York University,
4700 Keele Street,
Toronto, Ontario, Canada, M3J
1P3}
\email{ifarah@mathstat.yorku.ca}
\urladdr{http://www.math.yorku.ca/$\sim$ifarah}

\author[A. Vignati]{Alessandro Vignati}
\address[A. Vignati]{ Department of Mathematics, KU Leuven, Celestijnenlaan 200B, B-3001 Leuven, Belgium
}
\email{ale.vignati@gmail.com}
\urladdr{http://www.automorph.net/avignati}

\subjclass[2010]{}
\thanks{}
\date{\today}%
\maketitle

\begin{abstract}
In this paper, we study embeddings of uniform Roe algebras. Generally speaking, given metric spaces $X$ and $Y$, we are interested in which large scale geometric properties are stable under embedding of the uniform Roe algebra of $X$ into the uniform Roe algebra of~$Y$. 
\end{abstract}

\setcounter{tocdepth}{1}
\tableofcontents

\section{Introduction}
Given a metric space $X$, one can define its uniform Roe algebra, denoted by $\cstu(X)$. In a nutshell, the uniform Roe algebra of $X$ is the \cstar-subalgebra of $\cB(\ell_2(X))$  consisting of the norm closure of the algebra of finite propagation operators (see Definition \ref{DefUnifRoeAlg} for details). A version of this algebra was introduced by J. Roe in order to study the index theory of elliptical operators on noncompact manifolds  \cite{Roe1988,Roe1993}; later the study of these algebras was boosted     due to its intrinsic relation with the coarse Baum-Connes conjecture and, consequently, with the Novikov conjecture~\cite{Yu2000}. More recently, uniform Roe algebras (as well as Roe algebras) and its $\mathrm{K}$-theory have also been used as a framework in mathematical physics to study the classification of topological phases and  the topology of quantum systems -- in this setting, $\mathrm{K}_0(\cstu(X))$ can be interpreted as  the set of controlled topological phases \cite{EwertMeyer2019,Kubota2017}.

These notes deal with the problem of \emph{rigidity of uniform Roe algebras under embeddings}. Precisely, the rigidity problem for isomorphism of uniform Roe algebras -- i.e., the problem whether isomorphism between $\cstu(X)$ and $\cstu(Y)$ implies coarse equivalence  between $X$ and $Y$ -- has been studied by several authors \cite{BragaFarah2018,BragaFarahVignati2018,SpakulaWillett2013AdvMath}, and, under some geometric assumptions on the metric spaces, many positive results were obtained. Our goal in this paper is to study what can be said about the large scale geometry of a metric space $X$ given that $\cstu(X)$ embeds into the uniform Roe algebra of a metric space $Y$. Just as in the isomorphism case, some geometric assumptions   will be necessary. 

Let $X$ be a uniformly locally finite metric space (we refer the reader to  \S\ref{SectionPrelim} for definitions). An operator $a\in\cB(\ell_2(X))$ is a \emph{ghost} if for all $\varepsilon>0$ there exists a finite $A\subset X$ such that $|\langle a\delta_x,\delta_y\rangle|<\varepsilon$ for all $x,y\in X\setminus A$ -- here $(\delta_x)_{x\in X}$ denotes the standard unit basis of $\ell_2(X)$. Clearly, every compact operator is a ghost, and a uniformly locally finite metric space has \emph{property A} if and only if all ghosts in $\cstu(X)$ are compact \cite[Theorem 1.3]{RoeWillett2014}.

Property A represents the usual regularity condition given on $X$ when studying uniform Roe algebras. Notably, for uniformly locally finite   spaces, it is equivalent to amenability of the \cstar-algebra $\cstu(X)$ \cite[Theorem 5.3]{SkandalisTuYu2002}. In these notes we will often assume geometrical properties which are strictly weaker than property A itself. The main one relies on the following definition.

\begin{definition}
Let  $(X,d)$ be a metric space.
\begin{enumerate}
\item The space $X$ is  \emph{sparse} if there exists a partition $X=\bigsqcup_{n\in \N}X_n$ of $X$ into finite subsets such that $d(X_n,X_m)\to \infty$ as $n+m\to \infty$. 
\item The space  $X$ \emph{yields only compact ghost projections} if every ghost projection in $\cstu(X)$ is compact.
\end{enumerate}  
\end{definition}

In this paper we will often use the following condition:  \emph{all of the sparse subspaces of $X$ yield only compact ghost projections}. Although  technical, this property is quite general. The diagram below illustrates how it sits among some ``classic'' large scale properties. 
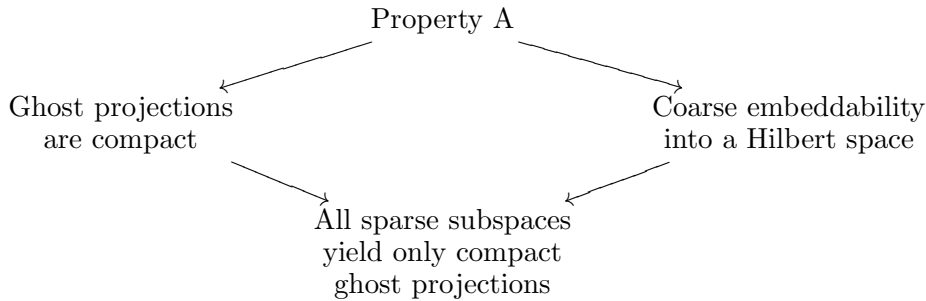
\begin{figure*}[h] \centerline{%
 \xymatrix@R-2ex{ 
         & \text{Property A}  \ar[dl]\ar[dr] &  \\
        \txt{Ghost projections \\ are compact} \ar[dr]& & \txt{Coarse  embeddability\\ into a Hilbert space}\ar[dl]\\
         & \txt{All sparse subspaces\\yield only compact\\ ghost projections } &   }
}
\caption{Relation between some geometric properties on a uniformly locally finite  metric space $X$.}
\end{figure*}

\noindent The arrow in the top left of the diagram above follows straightforwardly from the characterization of property A given above, and the top right arrow follows from \cite[Theorem 2.2]{Yu2000}. Both those arrows do not reverse, see \cite[Theorem 1.1]{ArzhantsevaGuentnerSpakula2012} and \cite[Theorem 1.4]{RoeWillett2014}. The bottom left arrow follows straight from the definitions, while the right bottom arrow is a consequence of \cite[Proposition 35]{FinnSell2014} (see \cite[Lemma 7.3]{BragaFarah2018} for details). We do not know whether the bottom arrows reverse.

In order to better understand our results on embeddings of uniform Roe algebras, it is important to recall the state-of-the-art regarding the rigidity of uniform Roe algebras under isomorphisms. It was proved in \cite[Theorem 4.1]{SpakulaWillett2013AdvMath} that if $\cstu(X)$ is isomorphic to $\cstu(Y)$, then $X$ and $Y$ are coarsely equivalent provided the metric spaces have  property A. This result was strengthened in \cite[Corollary 1.2 and Corollary 1.3]{BragaFarah2018}   for metric spaces which coarsely embed into a Hilbert space and for metric spaces yielding only compact ghost projections. Although not explicitly written, it follows from \cite{BragaFarah2018} that the same holds with the (formally) weaker assumption that all the sparse subspaces of the metric spaces yield only compact ghost projections.

A key ingredient for the rigidity results above  is to notice that, given an isomorphism  $\Phi\colon\cstu(X)\to \cstu(Y)$, there exists a unitary $U\colon \ell_2(X)\to \ell_2(Y)$ such that $\Phi(a)=UaU^*$ for all $a\in\cstu(X)$ (see \cite[Lemma 3.1]{SpakulaWillett2013AdvMath}). From this, one shows that (i) $\Phi$ is rank preserving, i.e., $\Phi$ sends minimal projections to minimal projections, and (ii) $\Phi$ is strongly continuous. Those properties play an essential role in the proofs of the rigidity theorems mentioned above. Unfortunately a nonsurjective $\Phi$ does not automatically satisfy neither (i) nor (ii), therefore we will need to substantially change our approach.

We now describe the main results in our paper. As we see below (Proposition \ref{PropositionEmbDoesNotImpCoarseEmb}), embeddability of  a uniform Roe algebra $\cstu(X)$ into $\cstu(Y)$ does not imply coarse embeddability of $X$ into $Y$ -- even under strong geometric conditions on $X$ and $Y$. 
However, assuming the embedding is compact preserving, i.e., $\Phi(a)$ is compact for all compact $a\in\cstu(X)$, we obtain the following (part \eqref{T12.2} is trivial, and in fact holds for all metric spaces $X$ and $Y$). 

\begin{theorem}\label{ThmRigidityUniformRoeAlgEmbeddingsFINITEUNION}
 Let $X$ and $Y$ be uniformly locally finite metric spaces and assume that all the sparse subspaces of $Y$ yield only compact ghost projections. 
 \begin{enumerate}
 \item \label{T12.2} If there exists an injective  coarse map $X\to Y$,  then $\cstu(X)$ embeds into $\cstu(Y)$. 
 \item \label{T12.1} If   $\cstu(X)$ embeds into $\cstu(Y)$ by a compact preserving map, then   there are $k\geq 1$ and  a partition $X=\bigsqcup_{n=1}^k X_n$ so that  for all $n\in \{1,\ldots, k\}$ there exists an injective  coarse map $X_n\to Y$.\footnote{Recall, a coarse map $f\colon (X,d)\to (Y,\partial)$ is a map so that for all $s>0$ there exists $r>0$ such that $d(x,y)<s$ implies $\partial (f(x),f(y))<r$. See \S\ref{SectionPrelim}.} 
 \end{enumerate}
 \end{theorem}

If we furthermore assume that the embedding $\cstu(X)\to \cstu(Y)$ in Theorem \ref{ThmRigidityUniformRoeAlgEmbeddingsFINITEUNION} \eqref{T12.1} is rank preserving, then there exists a coarse map $f\colon X\to Y$ which is  \emph{uniformly finite-to-one}, i.e., $\sup_{y\in Y}|f^{-1}(\{y\})|<\infty$ (see Theorem \ref{ThmRigidityUniformRoeAlgEmbeddings}).

The next corollary is the main motivation for Theorem \ref{ThmRigidityUniformRoeAlgEmbeddingsFINITEUNION}. It shows that we have an interesting large scale geometry preservation phenomenon happening in case of an embedding $\cstu(X)\to\cstu(Y)$.

\begin{corollary}\label{CorRigidityUniformRoeAlgEmbeddingsAsymDim}
Let $X$ and $Y$ be uniformly locally finite metric spaces and assume that    $\cstu(X)$  embeds into $\cstu(Y)$ by a compact preserving map. Then
\begin{enumerate}
\item\label{CorRigidityUniformRoeAlgEmbeddingsAsymDim.1}  $\mathrm{asydim}(X)\leq \mathrm{asydim}(Y)$, 
\item\label{CorRigidityUniformRoeAlgEmbeddingsAsymDim.1b} if $Y$ has property A, so does $X$, and
\item \label{CorRigidityUniformRoeAlgEmbeddingsAsymDim.2} if $Y$ has finite decomposition complexity (see \S\ref{SubsectionFDC}),  then so does $X$.
\end{enumerate}
\end{corollary}

Furthermore, we obtain stronger conclusions by strengthening the hypothesis on the structure of the embedding $\Phi\colon\cstu(X)\to \cstu(Y)$. As already noted, to apply strategies similar to the ones used in the isomorphism's setting, one has to assume that $\Phi$ is both rank preserving and strongly continuous. Such hypotheses are satisfied in case the image of $\Phi$ is a hereditary \cstar-subalgebra of $\cstu(Y)$ (see Lemma \ref{LemmaPhiStronglyContAndU}). 

\begin{theorem}\label{ThmRigidityUniformRoeAlgEmbeddingsHereditary}
Let $X$ and $Y$ be uniformly locally finite metric spaces such that $\cstu(X)$ is isomorphic to a hereditary \cstar-subalgebra of  $\cstu(Y)$. Then
\begin{enumerate}[label=(\roman*)]
\item\label{RigHer1} If all the sparse subspaces of $Y$ yield only compact ghost projections, then $X$ coarsely embeds into $Y$. 
\item\label{RigHer2} If, in addition, $Y$ has property A, then $X$ coarsely embeds into $Y$ by an injective map.
\end{enumerate}
\end{theorem}

Theorem \ref{ThmRigidityUniformRoeAlgEmbeddingsHereditary} can be used  to improve results on the rigidity of uniform Roe algebras under isomorphisms. Precisely, \cite[Corollary 1.2]{BragaFarah2018} shows that if $X$ and $Y$ are uniformly locally finite metric spaces with isomorphic uniform Roe algebras, then, if \emph{both} $X$ and $Y$ coarsely embed into a Hilbert space, $X$ and $Y$ are coarsely equivalent. The next corollary shows that one only needs to require \emph{one} of the metric spaces to coarsely embed into a Hilbert space.

\begin{corollary} \label{CorRigidityUniformRoeAlgebra}
Suppose that $X$ and $Y$ are  uniformly locally finite  metric spaces and    $Y$ coarsely embeds into a Hilbert space.
 If $\cstu(X)$ and $\cstu(Y)$ are isomorphic, then $X$ and $Y$ are coarsely equivalent.
\end{corollary}

This paper is organized as follows. \S\ref{SectionPrelim} is dedicated to notation and terminology. Also, we present an example of uniformly locally finite metric spaces $X$ and $Y$ so that $\cstu(X)$ embeds into $\cstu(Y)$ by a rank preserving strongly continuous $^*$-homomorphism but $X$ does not coarsely embed into $Y$ (Proposition \ref{PropositionEmbDoesNotImpCoarseEmb}), and we show that asymptotic dimension is stable under maps which are coarse and uniformly finite-to-one  (Proposition \ref{PropCoarseFiniteToOneAsyDim}). In \S\ref{SectionCoarseLikeQuasiLocal}, we continue the study  of coarse-like maps (see Definition \ref{Def.CoarseLike+}) introduced in \cite{BragaFarah2018,BragaFarahVignati2018} and show that if $Y$ has property A, then every strongly continuous linear map $\cstu(X)\to \cstu(Y)$ is coarse-like. For this, we introduce a quantitative version of \emph{quasi-local operators} (see Definition \ref{DefQuasiLocal}).

In \S\ref{SectionMakingStrongCont}, we show that one can often  assume that an embedding $\cstu(X)\to \cstu(Y)$ is strongly continuous (Theorem \ref{ThmEmbImpliesStrongEmb}). This result is essential in \S \ref{SectionEmbGeomPreserv}, where we prove Theorem \ref{ThmRigidityUniformRoeAlgEmbeddingsFINITEUNION} and Corollary \ref{CorRigidityUniformRoeAlgEmbeddingsAsymDim}. In \S\ref{SectionEmbHere}, we study isomorphisms 
between  uniform Roe algebras and hereditary subalgebras of uniform Roe algebras and prove Theorem \ref{ThmRigidityUniformRoeAlgEmbeddingsHereditary}.

Most of the results in this paper make use of the geometric property of all sparse subspaces yielding only compact ghost projection. In \S\ref{SectionGeomProp}, we discuss a variation of this property and show that it
 is a coarse invariant. At last, in \S\ref{SectionOpenProb}, we pose many natural questions which are left open.

\section{Preliminaries}\label{SectionPrelim}

\subsection{Uniform Roe algebra}
Given a Hilbert space $H$, we denote the space of bounded operators on $H$ by $\cB(H)$ and its ideal of compact operators by $\cK(H)$. Given a set $X$,  the Hilbert space of square summable complex-valued families indexed by $X$ is denoted by $\ell_2(X)$, and $(\delta_x)_{x\in X}$ denotes its canonical  basis. The \emph{support of $a\in \cB(\ell_2(X))$ } is defined as
\[\supp(a)=\{(x,y)\in X\times X\mid \langle a\delta_x,\delta_y\rangle\neq 0\}.\]
Let $\Delta_X=\{(x,x)\in X\times X\mid x\in X\}$. Then $\ell_\infty(X)$ is naturally identified with the subalgebra $\{a\in \cB(\ell_2(X))\mid \supp(a)\subseteq \Delta_X\}$ of $\cB(\ell_2(X))$. Given $x,y\in X$, we define an operator $e_{xy}\in \cB(\ell_2(X))$ by 
\[e_{xy}\delta_z=\langle \delta_z,\delta_x\rangle\delta_y\]
for all $z\in X$. Given $A\subseteq X$, we write $\chi_A=\sum_{x\in A}e_{xx}$.

Let $(X,d)$ be a metric space.   Given $r>0$ and $a\in \cB(\ell_2(X))$, we say that $a$ has \emph{propagation at most $r$}, denoted by $\propg(a)\leq r$, if $\langle a\delta_x,\delta_y\rangle=0$ for all $x,y\in X$ with $d(x,y)> r$.  An operator $a\in \cB(\ell_2(X))$ has \emph{finite propagation} if $\propg(a)\leq r$ for some $r>0$.
   
\begin{definition}\label{DefUnifRoeAlg}
Let $(X,d)$ be a metric space. 
  The \emph{uniform Roe algebra of $X$}, denoted by \emph{$\cstu(X)$}, is the closure of the operators in $\cB(\ell_2(X))$ of finite propagation. So, $\cstu(X)$ is a \cstar-algebra
\end{definition}

Results about uniform Roe algebras are often restricted to the class of uniformly locally finite metric spaces. Recall, a metric space $(X,d)$ is called \emph{uniformly locally finite}\footnote{This property for metric spaces is often called \emph{bounded geometry} in the literature.} (abbreviated \emph{u.l.f.} from now on) if $\sup_{x\in X}|B_r(x)|<\infty$ for all $r>0$. Note that all u.l.f. spaces are countable.

\subsection{Geometric  properties}

\begin{definition}
Let $n\in\N$. A metric space $(X,d)$ has \emph{asymptotic dimension at most $n$} if for all $R>0$ there exist $\cU_0,\ldots,\cU_n\subseteq \cP(X)$ such that 
\begin{enumerate}
\item $X=\bigcup_{i=0}^n\bigcup_{U\in \cU_i}U$, 
\item $d(U,U')>R$ for all $i\leq n$ and all distinct $U,U'\in \cU_i$, and
\item $\sup_{U\in \cU_i}\diam(U)< \infty$ for all $i\leq n$.
\end{enumerate}
\end{definition}

Although some of our results use Yu's property A, we do not make use of it per se. We use instead the operator norm localization property. As shown in \cite[Theorem 4.1]{Sako2014}, this property is equivalent to property A for u.l.f. metric spaces. For this reason, we chose not to present a definition of property A in these notes and we simply refer the interested reader to \cite[Definition 2.1]{Yu2000} (or \cite[Chapter 4]{NowakYuBook}). Given a vector $\xi\in \ell_2(X)$, write $\supp(\xi)=\{x\in X\mid \langle\xi,\delta_x\rangle\neq 0\}$.

\begin{definition}
A metric space $X$ has the \emph{operator norm localization property} (\emph{ONL} for short) if for all $r>0$ and all $\varepsilon>0$ there exists $s>0$ such that for every $a\in \cstu(X)$ with $\propg(a)\leq r$ there exists a unit vector $\xi \in \ell_2(X)$ with $\diam(\supp(\xi))\leq s$ so that $\|a\xi\|\geq (1-\varepsilon)\|a\|$.
\end{definition}

For the sake of curiosity, given a u.l.f.  metric space $X$, $X$ has property A if and only if $\cstu(X)$ is a nuclear \cstar-algebra \cite[Theorem 5.3]{SkandalisTuYu2002}.

\subsection{Coarse geometry of metric spaces}

Let $(X,d)$ and $(Y,\partial)$ be metric spaces and $f\colon X\to Y$ be a map.   The map $f$ is called \emph{coarse} if 
\[\sup\{\partial(f(x),f(y))\mid d(x,y)\leq r\}<\infty\]
for all $r>0$, and $f$ is called \emph{expanding} if 
\[\lim_{r\to \infty}\inf\{\partial(f(x),f(y))\mid d(x,y)\geq r\}=\infty.\]
If $f\colon X\to Y$ is both coarse and expanding, $f$ is said a \emph{coarse embedding}.

It is well known that asymptotic dimension, property $A$ and  finite decomposition complexity (FDC) are stable under coarse embeddings (the FDC is discussed in \S\ref{SubsectionFDC},  
see also \cite[Definition 2.7.3]{NowakYuBook}).
 Precisely, if $X$ coarsely embeds into $Y$, then $\mathrm{asydim}(X)\leq \mathrm{asydim}(Y)$ and, if $Y$ has Property A (FDC), then so does $X$ (see \cite[Proposition 2.2.4, Theorem 2.2.5, and Theorem 2.8.4]{NowakYuBook}).

\subsection{Embedding of uniform Roe algebras does not imply coarse embeddability}

We now show that the question about rigidity of uniform Roe algebras under embeddings has a negative answer. Mimicking the results obtained for isomorphisms (at least in the property A case, where two uniform Roe algebras are isomorphic if and only if the associated spaces are bijectively coarse equivalent, see \cite[Theorem 1.11]{BragaFarahVignati2018}), one is tempted to conjecture that $\cstu(X)$ embeds into $\cstu(Y)$ if and only if  $X$ coarsely embeds into $Y$. We now show that, while one direction always holds, if there is no assumption on the image of an embedding $\cstu(X)\to \cstu(Y)$, then one cannot  conclude that $X$ coarsely embeds into $Y$. First, we prove the trivial part of Theorem~ \ref{ThmRigidityUniformRoeAlgEmbeddingsFINITEUNION}.

\begin{proof}[Proof of Theorem \ref{ThmRigidityUniformRoeAlgEmbeddingsFINITEUNION} \eqref{T12.2}] Suppose $f \colon X\to Y$ is an injective coarse map. Define $\Phi\colon \cstu(X)\to\cB(\ell_2(Y))$ by
\[
\langle\Phi(a)\delta_y,\delta_{y'}\rangle=\begin{cases}
\langle a\delta_x,\delta_{x'}\rangle &\text{ if } f(x)=y\text{ and }f(x')=y'\\
0&\text{else}
\end{cases}.
\]
 Since $f$ is injective, so is $\Phi$. Fix $r<\infty$ and $R<\infty$ such that $d(x,x')\leq r$ implies $\partial (f(x), f(x'))\leq R$. 
If $a\in \cstu(X)$ has propagation $r$, then $\Phi(a)$ has propagation $R$, hence $\Phi$ maps the algebraic uniform Roe algebra of $X$ into the algebraic uniform Roe algebra of $Y$.    
Since $\Phi$ is a $^*$-homomorphism, it maps $\cstu(X)$ into $\cstu(Y)$. 
\end{proof} 

\begin{proposition}\label{PropositionEmbDoesNotImpCoarseEmb}
 Then there are u.l.f.  metric spaces $X$ and $Y$ with asymptotic dimension $1$ such that $\cstu(X)$ embeds into $\cstu(Y)$ by a strongly continuous rank preserving $^*$-homomorphism, but $X$ does not coarsely embed  into $Y$.
\end{proposition}

\begin{proof} Let $X=\Z$ and $Y=\N$ with the usual metrics. Then $f\colon \Z\to \N$ defined by $f(n)=2n$ if $n\geq 0$ and $f(n)=2|n|-1$ if $n<0$ is an injective coarse map, and therefore $\cstu(\Z)$ embeds into $\cstu(\N)$ by Theorem~ \ref{ThmRigidityUniformRoeAlgEmbeddingsFINITEUNION}\eqref{T12.2}. Note that the embedding is both strongly continuous and rank preserving. 
It is well-known that $\Z$ does not coarsely embed into $\bbN$: Otherwise, suppose $\varphi\colon \Z\to \N$ is a coarse embedding.  
Let $r<\infty$ be such that $|m-n|\leq 1$ implies $|\varphi(m)-\varphi(n)|\leq r$. 
Then for every $m\geq 0$ there exist $k\in \bbN$ such that $|\varphi(m)-k|\leq r$  and $n<0$ such that $|\varphi(n)-k|\leq r$. 
This implies that $\varphi$ cannot be expanding. 
%
%
\end{proof}

Although Proposition \ref{PropositionEmbDoesNotImpCoarseEmb} shows that rigidity of uniform Roe algebras under embeddings does not hold, the spaces 
$\N$ and $\Z$  are still very ``similar''. For instance, they both have asymptotic dimension equal to 1. See \S\ref{SectionOpenProb} for a discussion of related open problems.

The following generalizes the fact that asymptotic dimension is monotone with respect to coarse embeddings and it will be used to obtain Corollary~\ref{CorRigidityUniformRoeAlgEmbeddingsAsymDim}(\ref{CorRigidityUniformRoeAlgEmbeddingsAsymDim.1}).

\begin{proposition}\label{PropCoarseFiniteToOneAsyDim}
Let $(X,d)$ and $(Y,\partial)$ be metric spaces and assume that $(X,d)$ is u.l.f.. If there exists a uniformly finite-to-one coarse map $f\colon Y\to X$, then  \[\mathrm{asydim}(Y,\partial)\leq \mathrm{asydim}(X,d).\] 
In particular, if  $d\leq \partial$ then $\mathrm{asydim}(X,\partial)\leq \mathrm{asydim}(X,d)$.
\end{proposition}

\begin{proof}
If $ \mathrm{asydim}(X,d)=\infty$, the result is trivial. Suppose $ \mathrm{asydim}(X,d)\leq n$, for some $n\in\N$. Fix $R>0$. Since $f$ is coarse, there exists $S>0$ such that $d(f(y),f(y'))>S$ implies $\partial  (y,y')>R$. By the definition of asymptotic dimension, there are $\cU_0,\ldots,\cU_n\subseteq \cP(X)$ such that 
\begin{enumerate}
\item $X=\bigcup_{i=0}^n\bigcup_{U\in \cU_i}U$, 
\item $d(U,U')>S$, for all $i\leq n$, and all distinct $U,U'\in \cU_i$, and
\item $\sup_{U\in \cU_i}\diam (U)<\infty$ for all $i\leq n$.
\end{enumerate}
Since $(X,d)$ is u.l.f., there exists $N\in \N$ such that 
\[
\sup_{U\in \cU_i}|U|<N,\text{  for all }i\leq n.
\]
For each $i\leq n$ and $U\in\cU_i$, let $\tilde U=f^{-1}(U)\subseteq Y$. Since $f$ is uniformly finite-to-one, we have that 
\[
M=\sup_{i\leq n}\sup_{U\in \cU_i}|\tilde U|<\infty.
\]
\begin{claim}\label{Claim1}
Let $(E,d)$ be a metric space and $V\subseteq E$ with $|V|\leq M$. For all $R>0$, there exist $k\leq M$ and $V_1,\ldots,V_{k}\subseteq V$ such that $V=\bigcup_{i=1}^{k}V_{k}$, $\partial(V_i,V_j)>R$ for all $i\neq j$, and $\diam(V_i)\leq MR$ for all $i\leq k$.
\end{claim}

\begin{proof}
 If $\diam (V)\leq MR$, let $k=1$ and $V_1=V$.  If $\diam (V)> MR$, there exists a partition $V=V_1\sqcup V_2$ such that   $d(V_1,V_2)>R$. If  $\diam(V_i)\leq MR$ for $i\in \{1,2\}$, we are done. If not, continue this procedure until we are done. 
\end{proof}

For each $i\leq n$ and $U\in \cU_i$, let $k(i,U)$ and $V^U_1,\ldots, V^U_{k(i,U)}$ be given by Claim \ref{Claim1}. Let
\[
\cV_i=\{V^U_j\mid U\in \cU_i, j\leq k(i,U)\}.
\] 
So, $\sup_{V\in \cV_i}\diam (V)\leq MR$ for all $i\leq n$.  By our choice of $S$, it follows that  $\partial(V,V')>R$, for all $i\leq n$, and all distinct $V,V'\in \cV_i$. Since $R$ was arbitrary, this shows that  $\mathrm{asydim} (Y,\partial )\leq n$.

For the last statement, note that the identity map $(X,\partial)\to (X,d)$ is coarse if $d\leq\partial$.
\end{proof}

\section{Coarse-like maps and quasi-locality}\label{SectionCoarseLikeQuasiLocal}
A map between uniform Roe algebras is said coarse-like if, essentially, maps finite propagation operators  close to finite propagation operators in a uniform way (see Definition~\ref{Def.CoarseLike+}). Such maps were introduced formally in \cite[Definition 3.2]{BragaFarahVignati2018}, but were already studied in \cite{BragaFarah2018} where it was proved that strongly continuous linear maps $\Phi\colon\cstu(X)\to \cstu(Y)$ which are compact preserving must be coarse-like. The goal of this section is to discuss several results about coarse-like maps.

Given that operators in $\cstu(X)$ are arbitrarily close to finite propagation operators, the following definition,  implicit in \cite[Theorem 4.4]{BragaFarah2018} and formally introduced in \cite[Definition 3.2]{BragaFarahVignati2018}, helps us measuring quantitatively this fact.

\begin{definition} \label{Def.CoarseLike} 
Let $X$ be a metric space,  $\varepsilon>0$, and $k\in\N$. An operator $a\in \cB(\ell_2(X))$ can be  \emph{$\varepsilon$-$k$-approximated} if  there exists $b\in\cB(\ell_2(X))$ with propagation at most $k$ such that  $\|a-b\|\leq\varepsilon$. 
\end{definition} 

\begin{definition}\label{Def.CoarseLike+}
Let $X$ and $Y$ be  metric spaces, $\bA\subseteq \cstu(X)$ and $\Phi\colon \bA\to \cstu(Y)$ be a map. We say that $\Phi$ is \emph{coarse-like} if for all $m\in\N$ and all $\varepsilon>0$ there exists $k\in \N$ such that $\Phi(a)$ can be $\varepsilon$-$k$-approximated for every contraction $a\in \bA$ with $\propg(a)\leq m$.
\end{definition}

The following was essentially proved in \cite{BragaFarah2018} (compare this with \cite[Theorem 4.4]{BragaFarah2018} and \cite[Proposition 3.3]{BragaFarahVignati2018}).
For us $\D\subseteq\C$ is the closed unit disk.

\begin{proposition}\label{PropCoarseLike}
Let  $X$ and $Y$ be u.l.f. metric spaces,  and let $\Phi\colon \cstu(X)\to \cstu(Y)$ be a compact  preserving strongly continuous linear map.  Then $\Phi$ is coarse-like.
\end{proposition}

\begin{proof}
Fix $\varepsilon>0$, $m\in\N$, and let
\[E=\{(x,y)\in X\times X\mid d(x,y)\leq m\}.\] Since elements supported on $E$ have finite propagation, we have that
\[
\sum_{(x,x')\in E}\lambda_{x,x'}e_{xx'}=\mathrm{SOT}\text{-}\lim_{F\subseteq E, |F|<\infty}\sum_{(x,x')\in F}\lambda_{x,x'}e_{xx'}\in \cstu(X),
\]
 for every $\bar\lambda= (\lambda_{xx'})_{(x,x')\in E}\in \D^{E}$. The hypotheses on $\Phi$ imply that
\begin{enumerate}
\item $\Phi(e_{xx'})$ is a finite rank operator for all $x,x'\in X$, and
\item for all $\bar \lambda\in \D^{E}$, 
\[
b_{\bar \lambda}\coloneqq \Phi(\sum_{(x,x')\in E}\lambda_{x x'}e_{xx'})= \mathrm{SOT}\text{-}\lim_{F\subseteq E, |F|<\infty}\sum_{(x,x')\in F}\lambda_{x x'}\Phi(e_{xx'})
\]
belongs to $\cstu(Y)$.
 \end{enumerate}
 By \cite[Lemma 4.9]{BragaFarah2018} there is $k\in \N$ such that $b_{\bar \lambda}$ can be $\varepsilon$-$k$-approximated for all $\bar \lambda\in \D^E$. Since every contraction with propagation at most $m$ can be written as $\sum_{(x,x')\in E}\lambda_{x,x'}e_{xx'}$ for some  $\bar \lambda\in \D^{E}$, we conclude that $\Phi$ is coarse-like.
\end{proof}

For the remainder of this section, we work towards proving the following. 

\begin{theorem}\label{ThmCoarseLikeEvenForNonComp}
Let  $X$ and $Y$ be u.l.f. metric spaces,  and let $\Phi\colon \cstu(X)\to \cstu(Y)$ be a strongly continuous linear map.  If $Y$ has property A, then $\Phi$ is coarse-like.
\end{theorem}

Our proof of Theorem \ref{ThmCoarseLikeEvenForNonComp} relies on a recent result which states that if $X$ has property A, then operators in $\cstu(X)$ correspond to quasi-local operators \cite[Theorem 3.3]{SpakulaZhang2018}. Let us introduce this new concept as well as a quantitative version of it.

\begin{definition}\label{DefQuasiLocal}
Let $(X,d)$ be a metric space.
\begin{enumerate}
\item Let $\varepsilon>0$ and $s\in\N$. An operator $a\in \cB(\ell_2(X))$ is \emph{$\varepsilon$-$s$-quasi-local} if $\|\chi_Aa\chi_B\|\leq\varepsilon$ for all $A,B\subseteq X$ with $d(A,B)>s$.  
\item An operator $a\in \cB(\ell_2(X))$ is \emph{quasi-local} if for all $\varepsilon>0$ there exists $s>0$ such that $a$ is $\varepsilon$-$s$-quasi-local.
\end{enumerate}
\end{definition}

We need a more quantitative version of \cite[Theorem 3.3]{SpakulaZhang2018} for our goals, as we need to understand the link between $\varepsilon$-$s$-approximated operators and $\varepsilon$-$s$-quasi-local operators.

\begin{lemma}\label{LemmaQuasiLocalAndApproximable}
Let $X$ be a metric space with property A.  For all $\varepsilon>0$ and $r>0$, there exists $\delta=\delta(\varepsilon)>0$ and $m=m(\varepsilon,r)>0$ such that  every $\varepsilon$-$r$-quasi-local contraction in $\cB(\ell_2(X))$ can be $\delta$-$m$-approximated and $\lim_{\varepsilon\to 0^+}\delta(\varepsilon)=0$. Moreover, one can take $\delta(\varepsilon)=432(\varepsilon/9)^{1/3}$. 
\end{lemma}

\begin{proof}
This is implicit in the proofs of \cite[Theorem 2.8]{SpakulaTikuisis2019}, \cite[Lemma 6.3]{SpakulaWillett2017}, and \cite[Theorem 3.3]{SpakulaZhang2018}. Following \cite[\S 3]{SpakulaZhang2018},  given $\delta>0$ and $L>0$, let 
\[\mathrm{Comm}(L,\delta) =\{a\in \cB(\ell_2(X))\mid (b\in\ell_\infty(X),\Lip(b)\leq L)\Rightarrow \|[a,b]\| <\delta\norm{b}\},
\]
where $\Lip(b)$ denotes the Lipschitz constant of $b$ as a map $b\colon X\to \C$.

Fix $\varepsilon>0$, $r>0$ and an $\varepsilon$-$r$-quasi-local contraction $b\in \cB(\ell_2(X))$. The proof of the implication (ii)$\Rightarrow$(i) of \cite[Theorem 2.8]{SpakulaTikuisis2019}   with 
\[
\delta =18\Big(\frac{\varepsilon}{9}\Big)^{1/3}\text{ and } L= \frac{1}{2r}\Big(\frac{\varepsilon}{9}\Big)^{1/3}
\]
gives $b\in \mathrm{Comm}(L,\delta)$. Since $X$ has property A, \cite[Lemma 5.2]{SpakulaZhang2018}\footnote{In \cite[Lemma 5.2]{SpakulaZhang2018} $X$ is assumed to have the \emph{metric sparsification property}, but this is equivalent to property A for  u.l.f. metric spaces (see \cite[Propositioshn 4.1]{ChenTesseraWangYu2008}, \cite[Proposition 3.2 and Theorem 3.8]{BrodzkiNibloSpakulaWillettWright2013} and \cite[Theorem 4.1]{Sako2014}).} gives us that there exists $s>0$ such that for all $a\in \cB(\ell_2(X))$ with $\|a\|\leq 2$ and $a\in \mathrm{Comm}(L,2\delta)$, there exists $\xi\in \ell_2(X)$ with $\diam(\supp(\xi))\leq s$ and so that $\|a\xi\|\geq \|a\|-12\delta$. Since $X$ is u.l.f., $K=\sup_{x\in X}|B_x(s+1/L)|$ is finite. 

We now follow the proof of the implication (i)$\Rightarrow$(iv) of  \cite[Theorem 3.3]{SpakulaZhang2018}. Since $X$ has property A, \cite[Theorem 1.2.4]{Willett2009} implies that there exists a sequence of maps $(\varphi_i\colon X\to [0,1])_{i\in \N}$  such that
\begin{enumerate}
\item $\sup_{x \in X}|\{i\in \N\mid \varphi_i(x)\neq 0\}|<\infty$,
\item $\sup_{i\in \N}\diam(\supp(\varphi_i))<\infty$, 
\item $\sum_{i\in \N}\varphi_i(x)^2=1$ for all $x\in X$, and 
\item $\sum_{i\in \N}|\varphi_i(x)-\varphi_i(y)|^2<36\delta^2/K^2$ for all $x,y\in X$ with $d(x,y)\leq s+1/L$.
\end{enumerate}
The sequence $(\varphi_i)_{i\in \N}$ is called  a \emph{metric $2$-partition of unit   with $(s+1/L,6\delta/K)$-variation} (see \cite[Definition 6.1]{SpakulaWillett2017}). Considering $(\varphi_i)_{i\in \N}$ as a sequence of contractions in $\ell_\infty(X)$, \cite[Lemma 6.3]{SpakulaWillett2017} gives us that the sum  
\[b'=\sum_{i\in \N}\varphi_ib\varphi_i\]
converges in the strong operator topology to a bounded operator. By the proof of (i)$\Rightarrow$(iv) of  \cite[Theorem 3.3]{SpakulaZhang2018} we have that $\|b-b'\|\leq 24\delta$.

At last, it follows straightforwardly from the definition of $b'$ that 
\[\propg(b')\leq  \sup_{i\in \N}\diam(\supp(\varphi_i)).\]
Notice that $m=\sup_{i\in \N}\diam(\supp(\varphi_i))$ depends only on  $s$, $L$, $\delta$ and $K$, and therefore only on $\varepsilon$ and $r$. This finishes the proof.
\end{proof}

The following is essentially \cite[Lemma 4.7]{BragaFarah2018}. We include a proof for the reader's convenience.

\begin{lemma} \label{Lemma.eE2}
Suppose $X$ is a  metric space, $\e>0$, and $m\in \N$. Let $(a_i)_{i\in \N}\subseteq\cB(\ell_2(X))$ be a sequence which converges strongly to $a\in \cB(\ell_2(X))$. If  $a$ cannot be $\e$-$m$-approximated, then $a_i$ cannot be $\e$-$m$-approximated for all large enough $i$. 
\end{lemma} 

\begin{proof} Suppose otherwise. Then for every $i\in \N$ there exists $b_i\in \cB(\ell_2(X))$ with $\propg(b_i)\leq m$ such that $\|b_i-a_i\|\leq \varepsilon$. Let $M= \sup_{i\in\N}\|a_i\|$, so that $\|b_i\|\leq M+\e$ for all $i\in \N$. By going to a subsequence if necessary, we may assume that $(b_i)_{i\in\N}$ converges to some $b\in \cB(\ell_2(X))$ in the weak operator topology. Since for all $m\in\N$ the set of elements of propagation at most $ m$ is weakly closed, then  $\propg(b)\leq m$. Then $(a_i)_{i\in \N}$ and $(b_i)_{i\in\N}$ converge in the weak operator topology to $a$ and $b$, respectively, and $\norm{a_i-b_i}\leq\varepsilon$ for all $i\in\N$, which implies $\norm{a-b}\leq\varepsilon$.
 \end{proof} 

The next lemma should be thought as a tool to obtain contradictions. It will be used in the proof of Theorem \ref{ThmEmbImpliesStrongEmb}.
 
\begin{lemma}\label{LemmaSumOfQuasiLocalSOT}
Let $(X,d)$ be a metric space and $\{a_n\}_{n}\subseteq\cB(\ell_2(X))$ be quasi-local operators such that $\sum_{n\in M}a_n$ converges is the strong operator topology to a quasi-local element in $\cB(\ell_2(X))$ for all $M\subseteq \N$. Then for every  $\varepsilon>0$ there exists $n\in\N$ such that   $a_n$ is  $\varepsilon$-$n$-quasi-local.
\end{lemma}

\begin{proof}
Let $\varepsilon>0$ and assume for a contradiction that $a_n$ is not $2\varepsilon$-$n$-quasi-local for all $n\in\N$. We construct an increasing sequence $(m_i)_{i}\subseteq\N$ and sequences $(A_i)_i,(B_i)_i\subseteq X$ of finite subsets with 
\begin{enumerate}
\item $d(A_i,B_i)>m_i$ for all $i\in\N$,
\item $\|\chi_{B_{i}}a_{m_i}\chi_{A_{i}}\|>2\varepsilon$ for all $i\in\N$, and
\item $\|\chi_{B_{i}}a_{m_j}\chi_{A_{i}}\|< 2^{-j}\varepsilon$ for all $j\neq i$.
\end{enumerate}
We do so by induction. Let $m_1=1$. Since $a_{m_1}$ is not $2\varepsilon$-$m_1$-quasi-local, pick $A_1,B_1\subseteq X$ such that $d(A_1,B_1)>m_1$ and $\|\chi_{B_{1}}a_{m_1}\chi_{A_{1}}\|>2\varepsilon$. Without loss of generality, $A_1$ and $B_1$ are finite. Suppose $m_i, A_i$ and $B_i$ have been chosen for all $i\leq k$. Pick $s>0$ large enough so that $a_{m_i}$ is $(2^{-k-1}\varepsilon )$-$s$-quasi-local for all $i\in\{1,\ldots, k\}$. Let $F=\bigcup_{i=1}^k(A_i\cup B_i)$. Since $F$ is finite and $\sum_{i\in \N}a_i$ is strongly convergent, there exists $m_{k+1}>\max\{m_k,s\}$ such that  $\|a_{m_{k+1}}\chi_{F}\|<2^{-k-1}\varepsilon$. Since $a_{m_{k+1}}$ is not $2\varepsilon$-$m_{k+1}$-quasi-local, pick  $A_{k+1},B_{k+1}\subseteq X$  such that $d(A_{k+1},B_{k+1})>m_{k+1}$ and $\|\chi_{B_{k+1}}a_{m_{k+1}}\chi_{A_{k+1}}\|>2\varepsilon$. Again without loss of generality, $A_{k+1}$ and $B_{k+1}$ are finite, and this completes the induction.

By our choice of $(m_i)_i$, $(A_i)_i$ and $(B_i)_i$, we have that 
\[
\Big\|\chi_{B_i}\Big(\sum_ja_{m_j}\Big)\chi_{A_i}\Big\|\geq \|\chi_{B_j}a_{m_j}\chi_{A_j}\|-\sum_{j\in \N\setminus \{i\}}\|\chi_{B_i}a_{m_j}\chi_{A_i}\|
\geq \varepsilon
\]	
for all $i\in\N$. Since $\lim_i d(A_i,B_i)=\infty$, this implies that   $\sum_ja_{m_j}$ is not quasi-local, a contradiction.
\end{proof}

What  follows is a strengthening of \cite[Lemma 4.9]{BragaFarah2018} in the property A metric setting. 

\begin{lemma}\label{Lemma4.9ImprovedForPropA}
Let $X$ be a u.l.f. metric space with property A and let $(a_i)_{i\in \N}$ be a sequence of operators in $\cstu(X)$ such that $\sum_{i\in \N}\lambda_ia_i$ converges in the strong operator topology to an element in $\cstu(X)$ for all $(\lambda_i)_i\in \D^\N$. Then for all $\delta>0$ there exists $m\in \N$ such that $\sum_{i\in \N}\lambda_ia_i$ can be $\delta$-$m$-approximated for all $(\lambda_i)_i\in \D^\N$.
\end{lemma}

\begin{proof}
Suppose the result fails for $\delta>0$. So for all $m\in\N$ there exists $(\lambda_i)_i\in \D^\N$ such that  $\sum_{i\in \N}\lambda_ia_i$ cannot be $\delta$-$m$-approximated.

For each $n\in\N$ define
\[
W_n=\{(\lambda_i)_i\in \D^\N\mid \forall i\geq n(\lambda_i=0)\}
\]
and 
\[
Q_n=\{(\lambda_i)_i\in \D^\N\mid\exists \ell\in\N (\lambda_i\neq 0\Rightarrow i\in [n+1,\ell])\}.\]

\begin{claim}\label{Claim1.1.1}
For all $m,n\in\N$  there exists   $(\lambda_i)_i\in Q_n$  such that  $\sum_{i\in \N}\lambda_ia_i$ cannot be $\delta/2$-$m$-approximated. 
\end{claim}

\begin{proof}
Suppose the claim fails for $m,n\in\N$. Then $\sum_{i\in \N}\lambda_ia_i$ can be   $\delta/2$-$m$-approximated for all $(\lambda_i)_i\in Q_n$. Notice that 
\[\Big\{\sum_{i}\lambda_ia_i\in \cstu(X)\mid (\lambda_i)_i\in W_n\Big\}\] 
is compact in the norm topology. Hence, by the metric instance of \cite[Lemma 4.8]{BragaFarah2018}, we can pick $m_0>m$ such that  $\sum_{i\in \N}\lambda_ia_i$ can be   $\delta/2$-$m_0$-approximated for all $(\lambda_i)_i\in W_n$. Then  $\sum_{i\in \N}\lambda_ia_i$ can be   $\delta$-$m_0$-approximated for all $(\lambda_i)_i\in \D^\N$ which is eventually zero. By Lemma \ref{Lemma.eE2}, this implies that $\sum_{i\in \N}\lambda_ia_i$ can be   $\delta$-$m_0$-approximated for all  $(\lambda_i)_i\in \D^\N$;  contradiction.
\end{proof}

Fix $\varepsilon>0$ such that $432(\varepsilon/9)^{1/3}<\delta/2$ and for each $r\in \N$ let $m(\varepsilon,r)$ be given by Lemma \ref{LemmaQuasiLocalAndApproximable} (this is where property A is used).

\begin{claim}
For all $r,n\in\N$  there exists $(\lambda_i)_i\in Q_n$  such that  $\sum_{i\in \N}\lambda_ia_i$  is not $\varepsilon$-$r$-quasi-local.
\end{claim}
 
\begin{proof}
Fix $r,n\in\N$. Claim \ref{Claim1.1.1} gives $(\lambda_i)_i\in Q_n$ such that  $\sum_{i\in \N}\lambda_ia_i$ is not $\delta/2$-$m(\varepsilon,r)$-approximated. By Lemma \ref{LemmaQuasiLocalAndApproximable},  $\sum_{i\in \N}\lambda_ia_i$ is not $\varepsilon$-$r$-quasi-local.
\end{proof} 
 
By the claim, we can construct a sequence of disjoint intervals $(I_m)_m\subseteq\N$ and $(\lambda_i)_i\in \D^\N$ such that  $\sum_{i\in I_m}\lambda_i a_i$ is not $\varepsilon$-$m$-quasi-local for all $m\in\N$. On the other hand, $(\sum_{i\in I_m}\lambda_ia_i)_m$ satisfies the hypothesis of Lemma \ref{LemmaSumOfQuasiLocalSOT}, hence  $\sum_{i\in I_m}\lambda_i a_i$ is $\varepsilon$-$m$-quasi-local for some $m\in\N$;  contradiction.
\end{proof}

\begin{proof}[Proof of Theorem \ref{ThmCoarseLikeEvenForNonComp}]
Fix $n\in \N$ and  let \[E= \{(x,x')\in X\times X\mid d(x,x')\leq n\}.\] 
Since $\Phi$ is strongly continuous, the sum $\sum_{(x,x')\in E}\lambda_{xx'}\Phi(e_{xx'})$ converges in the strong operator topology to $\Phi(\sum_{(x,x')\in E}\lambda_{xx'}e_{xx'})$ for all $(\lambda_{xx'})_{(x,x')\in E}\in \D^E$. Lemma \ref{Lemma4.9ImprovedForPropA} says that  for every $\delta>0$ there exists $m\in\N$ such that $\Phi(\sum_{(x,x')\in E}\lambda_{xx'}e_{xx'})$ can be $\delta$-$m$-approximated for all $(\lambda_{xx'})_{(x,x')\in E}\in \D^E$. Since all contractions of propagation at most $ n$ belong to $\{ \sum_{(x,x')\in E}\lambda_{xx'}e_{xx'}\in \cstu(X)\mid (\lambda_{xx'})_{(x,x')\in E}\in \D^E\}$, this gives us that $\Phi$ is coarse-like.
\end{proof}

\section{Making embeddings strongly continuous}\label{SectionMakingStrongCont}

The main result of this section is Theorem \ref{ThmEmbImpliesStrongEmb}. When working with embeddings $\Phi\colon \cstu(X)\to \cstu(Y)$, it is often useful to assume that $\Phi$ is strongly continuous. This is a strenghtening of injectivity: since $\mathcal K(\ell_2(X))$ is a  minimal ideal strongly dense ideal in $\cstu(X)$, all strongly continuous (nonzero) $^*$-homomorphisms between uniform Roe algebras are injective. The following adaptation of \cite[Example~3.2.1]{FarahBook2000} shows that the converse does not hold.
\begin{proposition}
Let $X=\{n^2\mid n\in\N\}$. There exists an embedding $\cstu(X)\to\cstu(X)$ which is not strongly continuous.
\end{proposition}

\begin{proof}
Let $X_1=\{(2n)^2\mid n\in\N\}$, $X_2=\{(2n+1)^2\mid n\in\N\}$, and  $\cU$ be a nonprincipal ultrafilter on $\N$. Let $f\colon X\to X_1$ be the map $f(n^2)=4n^2$, so $f$ is a bijective coarse equivalence. Hence, $f$ induces an isomorphism $\Phi_1\colon \cstu(X)\to \cstu(X_1)$ (e.g., \cite[Theorem 8.1]{BragaFarah2018}). We now define a $^*$-homomorphism  $\Phi_2\colon \cstu(X)\to \cstu(X_2)$ as follows. For each $A\subseteq X$ let 
\[\Phi_2(\chi_A)=\left\{\begin{array}{ll}
\chi_{X_2},& \text{ if }A\in\cU,\\
0,&\text{ if }A\not\in \cU.
\end{array}\right.\] 
Since $\cU$ is an ultrafilter, $\Phi_2$ extends to a $^*$-homomorphism $\Phi_2\colon\ell_\infty(X)\to \ell_\infty(X_2)$. Since $\cU$ is nonprincipal and $\cstu(X)=\cst(\ell_\infty(X),\cK(\ell_2(X)))$, defining $\Phi_2(a)=0$ for all $a\in \cK(\ell_2(X))$, we can extend $\Phi_2$ to a $^*$-homomorphism $\Phi_2\colon \cstu(X)\to \ell_\infty(X_2)\subseteq \cstu(X_2)$.
As $\cU$ is nonprincipal, $\Phi_2(\chi_F)=0$ for all finite $F\subseteq\N$. As $\Phi_2(1)\neq 0$, $\Phi_2$ is not strongly continuous. Let 
\[
\Phi=\Phi_1\oplus\Phi_2\colon \cstu(X)\to \cstu(X_1)\oplus\cstu(X_2)\subseteq \cstu(X).
\]
 $\Phi$ is an embedding which is not strongly continuous. 
\end{proof}
The following is a straightforward consequence of \cite[Lemma 3.4]{BragaFarahVignati2018}\footnote{More precisely, in \cite[Lemma 3.4]{BragaFarahVignati2018} one has $\cstu(Y)$ instead of $\cB(\ell_2(Y))$. However, this is not used in the proof.}.

\begin{lemma} \label{L:AA} 
Suppose  $X$ and $Y$ are countable metric spaces  and  $\Phi\colon \ell_\infty(X)\to \cB(\ell_2(Y))$ is   a strongly continuous $^*$-homomorphism.  Then for every  finite  $F\subseteq Y$  and  every $\e>0$ there exists a finite  $E\subseteq X$ such that  $\norm{\chi_F\Phi(\chi_A)}<\e$ for all $A\subseteq X\setminus E$.
\end{lemma}

We can now prove the main result of this section.

\begin{theorem}\label{ThmEmbImpliesStrongEmb}
Let  $X$ and $Y$ be  u.l.f. metric spaces,  and  $\Phi\colon \cstu(X)\to \cstu(Y)$ be an embedding. Assume that either 
\begin{enumerate}[label=(\roman*)]
\item $\Phi$ is compact preserving, or
\item $Y$ has property A.
\end{enumerate}
 Then, there exists a projection $p\in \cstu(Y)$  in the commutator of $\Phi(\cstu(X))$ such that  
 \[
 a\in \cstu(X)\mapsto p\Phi(a)p\in \cstu(Y)
 \] 
is a nonzero strongly continuous $^*$-homomorphism.
\end{theorem}

\begin{proof}
 Let $(X_n)_n$ be an increasing sequence of finite subsets of $X$ so that $X=\bigcup_nX_n$. Since $(\Phi(\chi_{X_n}))_n$ is an increasing sequence of projections bounded above by the projection $\Phi(1)$, we can define a projection
\[p=\mathrm{SOT}\text{-}\lim_n\Phi(\chi_{X_n}).\]
Since $\Phi$ is injective, $p\neq 0$ and $\Phi(1)\geq p$.

\begin{claim}
$p$ commutes with $\Phi(\cstu(X))$.
\end{claim}

\begin{proof}
First  notice that $p$ commutes with $\Phi(\cK(\ell_2(X)))$. For this, it is enough to show that $p$ commutes with the image $\cB(\ell_2(X_n))$, for all $n$. Fix then $a=\chi_{X_n}a\chi_{X_n}=a$. Since $p\Phi(\chi_{F})=\Phi(\chi_F)$ whenever $F\subseteq X$ is finite, we are done.

We now show the general statement. Let $a\in \cstu(X)$. By the previous paragraph, $\Phi(a\chi_{X_n})p=p\Phi(a\chi_{X_n})$ for all $n\in\N$. Hence, it follows that 
\[
\Phi(a)p=\Phi(a)p^2=\mathrm{SOT}\text{-}\lim_n\Phi(a\chi_{X_n})p=\mathrm{SOT}\text{-}\lim_np\Phi (a\chi_{X_n})=p\Phi(a)p
\]
Similarly, we have that $p\Phi(a)=p\Phi(a)p$, so $p\Phi(a)=\Phi(a)p$, and we are done.
\end{proof}
Let $\Phi^p=p\Phi p$. By the claim, $\Phi^p$ is a $^*$-homomorphism.

\begin{claim}
$\Phi^p$ is strongly continuous.
\end{claim}

\begin{proof}
Let $a\in \cstu(X)$  and $(a_k)_k$ be a sequence in $\cstu(X)$ converging to $a$ in the strong operator topology. We can assume $a$ and each $a_k$ are contractions. We want to show that for every $\xi\in\ell_2(Y)$ we have that  
\[
\Phi^p(a)\xi=\lim_k\Phi^p(a_k)\xi.
\]
Let $H=\mathrm{Im}(p)$. Since $\Phi^p(\cstu(X))[H^{\perp}]=0$, we can assume that $\xi\in H$. Let $\varepsilon>0$, and let $\ell\in \N$ be large enough so that $\|\Phi(\chi_{X_\ell})\xi-\xi\|<\varepsilon$.
 Since $X_\ell$ is finite, $a\chi_{X_\ell}=\lim_ka_k\chi_{X_\ell}$, which implies that $\Phi^p(a \chi_{X_\ell})=\lim_k\Phi^p(a_k\chi_{X_\ell})$. Since $\Phi$ and $\Phi^p$ agree on $\mathcal K(\ell_2(X))$, we have that  
\[
\limsup_k\|\Phi^p(a_k)\xi-\Phi^p(a)\xi\|\leq \lim_k\|\Phi^p(a_k \chi_{X_\ell})\xi-\Phi^p(a \chi_{X_\ell})\xi\|+2\varepsilon=2\varepsilon.
\]
Since $\varepsilon$ is arbitrary we have the thesis.
\end{proof}

\begin{claim}\label{ClaimcccOverFin}
$\SJ=\{A\subseteq X\mid \Phi(\chi_A)=\Phi^{p}(\chi_A)\}$ is a ccc/Fin ideal\footnote{An ideal $\SI\subseteq\mathcal P(X)$ is ccc/Fin if every family of almost disjoint sets which are not in $\SI$ is countable.}.
\end{claim}

\begin{proof}
As $\Phi(a)=p\Phi(a)p+(1-p)\Phi(a)(1-p)$ for all $a\in\cstu(X)$ we have that $\SJ=\{A\subseteq X\mid \Phi^{1-p}(\chi_A)=0\}$. Let $(A_i)_{i\in I}$ be a uncountable family of infinite subsets of $X$ which is almost disjoint, i.e., $A_i\cap A_j$ is finite for all $i\neq j$. Since $\Phi^{1-p}(\chi_F)=0$ for all finite $F\subseteq X$ and $\Phi^{1-p}$ is a $^*$-homomorphism, we have that $(\Phi^{1-p}(\chi_{A_i}))_{i\in I}$ is an uncountable family of orthogonal projections on $\ell_2(Y)$. Since $Y$ is countable, this implies that $\{i\in I\mid \Phi^{1-p}(\chi_{A_i})= 0\}$ is uncountable, in particular, nonempty.  
\end{proof}

We are left to show that $p\in\cstu(Y)$ if $\Phi$, or $Y$, are as in the hypotheses.

(i) Given a u.l.f. metric space $Z$, we define the \emph{uniform Roe corona of $Z$} by $\roeq(Z)=\cstu(Z)/\cK(\ell_2(Z))$ and let $\pi_Z\colon \cstu(Z)\to\roeq(Z)$ denote the quotient map.\footnote{We refer the reader to \cite{BragaFarahVignati2018} for a detailed study of uniform Roe coronas.} If $\Phi(\cK(\ell_2(X)))\subseteq\cK(\ell_2(Y))$, then $\Phi$ induces a $^*$-homomorphism 
\[
\tilde\Phi\colon \roeq(X)\to \roeq(Y),
\]
where $\Phi^p\restriction \ell_\infty(X)$ lifts $\tilde \Phi\restriction \ell_\infty(X)/c_0(X)$ on $\SJ$, i.e., 
\[
\tilde\Phi\circ\pi_X(\chi_A)=\pi_Y\circ\Phi^p(\chi_A)
\] 
for all $A\in \SJ$. Since $\SJ$ is ccc/Fin, \cite[Proposition 3.3]{BragaFarahVignati2018} implies that  $\Phi^p\restriction \ell_\infty(X)$ is coarse-like, hence the image of $\Phi^p$  is contained in $\cstu(Y)$, and so is $p=\Phi^p(1)$.

(ii) Suppose $Y$ has property A and suppose that $p=\Phi^p(1)\not\in \cstu(Y)$. By \cite[Theorem 3.3]{SpakulaZhang2018}, there exists $\varepsilon>0$ so that for all $s>0$ there exist $B,D\subseteq Y$ such that $\partial(B,D)>s$ and $\|\chi_Dp\chi_B\|>2\varepsilon$. Without loss of generality, $B$ and $D$ can be taken to be finite.

\begin{claim}\label{ClaimClaim1}
For all finite  $E_0\subseteq X$ and all $s>0$ there exists a finite $E\subseteq X\setminus E_0$ such that $\Phi^p(\chi_{ E})$ is not $\varepsilon$-$s$-quasi-local.
\end{claim}

\begin{proof}
If not,  fix offenders $E_0\subseteq X$ and $s>0$.   So, $\|\chi_D\Phi^p(\chi_{ E})\chi_B\|\leq \varepsilon$ for all finite $E\subseteq X\setminus E_0$ and all $B,D\subseteq Y$ with $d_Y(B,D)>s$. Since $E_0$ is finite, $\Phi^p(\chi_{E_0})\in \cstu(Y)$ and, by enlarging $s$ if necessary,  we can assume that $\Phi^p(\chi_{E_0})$ is $\varepsilon$-$s$-approximated. Let $\text{Fin}(X\setminus E_0)=\{E\subseteq X\setminus E_0\mid |E|<\infty\}$. Since $\Phi^p$ is strongly continuous and $p=\Phi^p(1)$, it follows that 
\[\|\chi_Dp\chi_B\|\leq\sup_{E\in\text{Fin}(X\setminus E_0)} \Big(\|\chi_D\Phi^p(\chi_{E_0})\chi_B\|+\|\chi_D\Phi^p(\chi_{  E})\chi_B\|\Big)\leq 2\varepsilon\]
for all finite $B,D\subseteq Y$ with $d_Y(B,D)>s$; contradiction.
\end{proof}

By the previous claim, we can pick a disjoint sequence $(E_n)_n$ of finite subset of $X$ such that $(\Phi^p(\chi_{E_n}))_n$ is not $\varepsilon$-$n$-quasi-local for all $n\in\N$. However, since $(\Phi^p(\chi_{E_n}))_n$ satisfies the hypothesis of Lemma \ref{LemmaSumOfQuasiLocalSOT}, so $\Phi^p(\chi_{E_n})$ is  $\varepsilon$-$n$-quasi-local for some $n\in\N$; contradiction.
\end{proof}

The following is a trivial consequence of Theorem \ref{ThmEmbImpliesStrongEmb}.

\begin{corollary}\label{CorMakeCompPreservEmbStrongCont}
 Let $X$ and $Y$ be u.l.f. metric spaces. The following holds.
 \begin{enumerate}
 \item If $\cstu(X)$ embeds into $\cstu(Y)$ by a compact preserving map, then $\cstu(X)$ embeds into $\cstu(Y)$ by a compact preserving strongly continuous map.
 \item If $Y$ has property A and  $\cstu(X)$ embeds into $\cstu(Y)$, then $\cstu(X)$ embeds into $\cstu(Y)$ by a strongly continuous map. \qed
 \end{enumerate}
\end{corollary}

\begin{remark}
If the assumption of property A is not necessary in   Theorem~\ref{ThmCoarseLikeEvenForNonComp} then  every strongly continuous $^*$-homomorphism $\ell_\infty\to\cstu(Y)$ is coarse-like. Given this, the projection $p$ defined in Theorem~\ref{ThmEmbImpliesStrongEmb} would automatically belong to $\cstu(Y)$. For  this, it is enough to show the following weak version of Lemma~\ref{LemmaSumOfQuasiLocalSOT}: If $a_i$ are orthogonal projections in $\cstu(Y)$ such that for each $(\lambda_i)_{i\in\N}\in\mathbb D^\N$ the element $\sum\lambda_ia_i\in\cstu(Y)$, then for every $\varepsilon>0$ there is $m$ such that each $a_i$ can be $\varepsilon$-$m$-approximated. We were unable to prove this  without property A.
See also Problem~\ref{ProblemPropAEmbHilb} and the paragraph following it. 
\end{remark}

\subsection{Almost coarse-like maps}\label{SubsectionAlmostCoarseLike}
In \S\ref{SectionCoarseLikeQuasiLocal}, we studied the coarse-like property for strongly continuous maps $\Phi\colon\cstu(X)\to \cstu(Y)$. In this subsection, we introduce a weaker property, the almost coarse-like property, and prove some results about it. Unlike in the coarse-like setting, this weaker property allows us to forget the strong continuity condition.

\begin{definition}\label{Def.AlmostCoarseLike}
Let $X$ and $Y$ be  metric spaces, $\bA\subseteq \cstu(X)$ and $\Phi\colon \bA\to \cstu(Y)$ be a map. We say that $\Phi$ is \emph{almost coarse-like} if for all $m\in\N$ and all $\varepsilon>0$ there exists $k\in \N$ such that $\Phi(a)$ can be $\varepsilon$-$k$-approximated for every contraction $a\in \bA$ with finite support and so that $\propg(a)\leq m$.
\end{definition}

\begin{proposition}\label{PropAlmostCoarseLike}
Let $X$ and $Y$ be u.l.f. metric spaces and  $\Phi\colon \cstu(X)\to \cstu(Y)$ be a compact preserving  $^*$-homomorphism. Then $\Phi$ is almost coarse-like.
\end{proposition}

\begin{proof}
If $\Phi$ is noninjective, then the kernel of $\Phi$ contains all compact operators, and in particular all operator of finite-dimensional range. Therefore $\Phi(a)=0$ whenever $a$ has finite-dimensional range, and $\Phi$ is obviously almost coarse-like. If $\Phi$ is injective, let $X_n$ be finite sets with $X_n\subseteq X_{n+1}$ and $\bigcup X_n=X$. Let $p=\mathrm{SOT}\text{-}\lim\Phi(\chi_{X_n})$, and $\Phi^p(a)=p\Phi(a)p$ for all $a\in \cstu(X)$.  By Theorem~\ref{ThmEmbImpliesStrongEmb},  $\Phi^p$ is strongly-continuous. Hence,  Proposition~\ref{PropCoarseLike} implies that $\Phi^p$ is coarse-like.  Since  $\Phi^p$ and $\Phi$ agree on the set of finite support elements, the thesis follows.
\end{proof}

The proof of the next proposition is analogous to the proof of Proposition \ref{PropAlmostCoarseLike}, but with Theorem \ref{ThmCoarseLikeEvenForNonComp} being used instead of Proposition \ref{PropCoarseLike}.

\begin{proposition}\label{PropAlmostCoarseLikePropA}
Let $X$ and $Y$ be u.l.f. metric spaces and  $\Phi\colon \cstu(X)\to \cstu(Y)$ be a    $^*$-homomorphism. If $Y$ has property A, then $\Phi$ is almost coarse-like.\qed
\end{proposition}

\section{Embeddings and geometry preservation}\label{SectionEmbGeomPreserv}

In this section, we prove the nontrivial part of Theorem \ref{ThmRigidityUniformRoeAlgEmbeddingsFINITEUNION} and as a consequence  obtain Corollary \ref{CorRigidityUniformRoeAlgEmbeddingsAsymDim}.

\begin{lemma}\label{LemmaPickMapf}
Let $X$ and $Y$ be  metric spaces and $\Phi\colon \cstu(X)\to\cstu(Y)$ be a compact preserving   embedding. Assume that every sparse subspace of $Y$ yields only compact ghost projections. Then 
\[\delta=\inf_{x\in X}\sup_{y\in Y}\|\Phi(e_{xx})e_{yy}\|>0.\]
Moreover, if $f\colon X\to Y$ is a map such that $\|\Phi(e_{xx})e_{f(x)f(x)}\|\geq \delta$ for all $x\in X$, then $f$ is $\lfloor \delta^{-2}\rfloor$-to-one. 
\end{lemma}

\begin{proof}
Since $X$ is countable, enumerate $X$, say $X=\{x_n\mid n\in\N\}$. By Claim~\ref{ClaimcccOverFin}, the set
\[\SI=\Big\{L\subseteq \N\mid \Phi\Big(\sum_{n\in L}	e_{x_nx_n}\Big)=\mathrm{SOT}\text{-}\sum_{n\in L}\Phi(e_{x_nx_n})\Big\}\]
is a ccc/Fin ideal.  For each $n\in\N$, let $p_n=\Phi(e_{x_nx_n})$. The result now follows straightforwardly from     \cite[Proposition 4.1]{BragaFarahVignati2018} applied to the sequence of finite rank projections $(p_n)_n$.\footnote{Precisely, the proof of \cite[Proposition 4.1]{BragaFarahVignati2018} starts by taking a an infinite $L\subseteq \N$ and working with the subsequence $(p_n)_{n\in L}$. Since $\SI$ is ccc/Fin, going  to a further subsequence we can suppose  that $L\in \SI$. The proof now holds verbatim.}

Let $f\colon X\to Y$ be a map such that $\|\Phi(e_{xx})e_{f(x)f(x)}\|\geq \delta$ for all $x\in X$. Let $y\in f(X)$ and $F\subseteq X$ be such that $f(x)=y$ for all $x\in F$. Since $(\Phi(e_{xx})\delta_y)_{x\in X}$ are orthogonal, it follows that 
\[1\geq\Big\|\sum_{x\in F}\Phi(e_{xx})\delta_y\Big\|^2=\sum_{x\in F}\|\Phi(e_{xx})\delta_y\|^2\geq \delta^2|F|.\]
Since $\sum_{x\in F}\Phi(e_{xx})$ is a contraction, this implies that $|F|\leq \delta^{-2}$.
\end{proof}

\begin{proof}[Proof of Theorem \ref{ThmRigidityUniformRoeAlgEmbeddingsFINITEUNION} \eqref{T12.1}]
By Corollary \ref{CorMakeCompPreservEmbStrongCont}, there exists a compact preserving strongly continuous embedding $\Phi\colon \cstu(X)\to \cstu(Y)$. Let $\delta>0$ and $f\colon X\to Y$ be the uniformly finite-to-one map given by Lemma \ref{LemmaPickMapf}. 

Fix  $x_0\in X$. Since $\Phi$ is compact preserving, $\text{Im}(\Phi(e_{x_0x_0}))$ is finite dimensional. By our choice of $\delta $ and $f$, $\|\Phi(e_{xx})\delta_{f(x)}\|>\delta$ for all $n\in\N$. Therefore, since $\Phi(e_{xx})\delta_{f(x)}\in \text{Im}(\Phi(e_{xx}))$ and $\Phi(e_{xx_0})\restriction\text{Im}(\Phi(e_{xx}))$ is an isometry onto $\text{Im}(\Phi(e_{x_0x_0}))$, we have that 
\[\Phi(e_{xx_0})\delta_{f(x)}\in D=\big\{\xi\in \ell_2(Y)\mid  \|\xi\|\in [\delta,1]\big\}\cap \text{Im}(\Phi(e_{x_0x_0}))\]
for all $x\in X$. By  compactness of $D$,  there exist $k\in\N$ and a partition $D=\bigsqcup_{n=1}^kD_n$ such that $\diam(D_n)<\delta$ for all $n\in\{1,\ldots,k\}$. For each $n\in\{1,\ldots, k\}$, let
\[X_n=\{x\in X\mid \Phi(e_{xx_0})\delta_{f(x)}\in D_n\}.\]
Clearly, $X=\bigsqcup_{n=1}^k X_n$ and 
\begin{equation}\label{TagAst1}
|\langle \Phi(e_{x_1x_0})\delta_{f(x_1)},\Phi(e_{x_2x_0})\delta_{f(x_2)}\rangle|>\frac{\delta^2}{2}\tag{$*$}
\end{equation}
for all $n\in\{1,\ldots,k\}$ and all $x_1,x_2\in X_n$.

\begin{claim}
For each $n\in\{1,\ldots,k\}$, the map $f\restriction X_n$ is coarse.
\end{claim}

\begin{proof}
Fix $n\in \{1,\ldots, k\}$ and let $r>0$. By Proposition \ref{PropAlmostCoarseLike}, there exists $s>0$ such that $\Phi(e_{x_1x_2})$ is $\delta^2/2$-$s$-approximated for all $x_1,x_2\in X$ with $d(x_1,x_2)<r$. In particular, for all $x_1,x_2\in X$ with $d(x_1,x_2)<r$, it follows that
\begin{equation}\label{TagAst2}
|\langle \Phi(e_{x_1x_2})\delta_{y_1},\delta_{f(y_2)}\rangle|\leq\frac{\delta^2}{2}\tag{$**$}
\end{equation}
 for all $y_1,y_2\in Y$ so that $\partial (y_1,y_2)>s$. Therefore, if $x_1,x_2\in X_n$ and $d(x_1,x_2)<r$, \eqref{TagAst1} and  \eqref{TagAst2} implies that $\partial (f(x_1),f(x_2))<s$. 
\end{proof}
Since $f$ is uniformly finite-to-one, by splitting each $X_n$ into finite many pieces if necessary, we can assume that $f\restriction X_n$ is injective for all $n\in\{1,\ldots,k\}$.
\end{proof}

\begin{proof}[Proof of Corollary \ref{CorRigidityUniformRoeAlgEmbeddingsAsymDim}\eqref{CorRigidityUniformRoeAlgEmbeddingsAsymDim.1} and \eqref{CorRigidityUniformRoeAlgEmbeddingsAsymDim.1b}]
\eqref{CorRigidityUniformRoeAlgEmbeddingsAsymDim.1}:
If $\mathrm{asydim}(Y)=\infty$, we are done. So assume that $Y$ has finite asymptotic dimension.  In particular, $Y$ has property A \cite[\S11.5]{RoeBook}, so its   subspaces yield only compact ghost projections. Hence, by Theorem \ref{ThmRigidityUniformRoeAlgEmbeddingsFINITEUNION}, there exist a partition $X=\bigsqcup_{n=1}^kX_n$ so that $X_n$ is mapped into $Y$ by a uniformly finite-to-one coarse map  for all $n\in\{1,\ldots, k\}$. By Proposition \ref{PropCoarseFiniteToOneAsyDim}, we have that  $\mathrm{asydim}(X_n)\leq \mathrm{asydim}(Y)$ for all $n\in\{1,\ldots,k\}$. Hence,  $\mathrm{asydim}(X)\leq \mathrm{asydim}(Y)$ (see \cite[Corollary 26]{BellDranishnikov2008}).

\eqref{CorRigidityUniformRoeAlgEmbeddingsAsymDim.1b}: Note that if $X=\bigsqcup_{n=1}^k X_n$, then $X$ has property $A$ if and only if each $X_n$ does; this follows for example from \cite[Lemma 4.1]{RoeWillett2014}. By Theorem~\ref{ThmRigidityUniformRoeAlgEmbeddingsFINITEUNION}, there exist a partition $X=\bigsqcup_{n=1}^kX_n$ so that $X_n$ is mapped into $Y$ by an injective coarse map $f_n\colon X_n\to Y$, for all $n\in\{1,\ldots, k\}$. Suppose that there is $n$ such that $X_n$ does not have property $A$, and  let $a\in\cstu(X_n)$ be a noncompact ghost. Define $b\in\cB(\ell_2(Y))$ by 
\[
\langle b\delta_y,\delta_{y'} \rangle=\begin{cases}\langle a\delta_x, \delta_{x'}\rangle&\text{ if } y=f_n(x)\text{ and } y'=f_n(x')\\
0&\text{ else.}
\end{cases}
\]
Then $b$ is a  noncompact ghost. Moreover, since $f_n$ is coarse, $b\in\cstu(Y)$, a contradiction to $Y$ having property $A$.
\end{proof}

\subsection{Rank preserving embeddings}
We now improve Theorem \ref{ThmRigidityUniformRoeAlgEmbeddingsFINITEUNION} in the case the embedding $\cstu(X)\to \cstu(Y)$ satisfies the stronger property of sending minimal projections to minimal projections. First, we need the following lemma.

\begin{lemma}\label{LemmaTheMapsAreCoarse}
Let $X$ and $Y$ be u.l.f. metric spaces and $\Phi\colon \cstu(X)\to \cstu(Y)$ be a rank preserving $^*$-homomorphism.  Then for all $r,\delta>0$  there exists $s>0$ such that for all $x_1,x_2\in X$ and all $y_1,y_2\in Y$ we have that if  $d(x_1,x_2)\leq r$, 
$\|\Phi(e_{x_1x_1})e_{y_1y_1}\|\geq \delta$, and $\|\Phi(e_{x_2x_2})e_{y_2y_2}\|\geq \delta$, then $\partial (y_1,y_2)\leq s$.
\end{lemma}

\begin{proof}
Suppose otherwise. Then there exist 
 $r,\delta>0$, 
 sequences $(x^1_n)_n$ and $(x^2_n)_n$ in $X$, and sequences $(y^1_n)_n$ and $(y^2_n)_n$  in $Y$ such that $d(x_n^1,x_n^2)\leq r$, 
$\|\Phi(e_{x_n^1x_n^1})e_{y_n^1y_n^1}\|\geq \delta$,  $\|\Phi(e_{x_n^2x_n^2})e_{y_n^2y_n^2}\|\geq \delta$, and 
$\partial(y_n^1,y_n^2)\geq n$ for all $n\in\N$.

By Proposition \ref{PropAlmostCoarseLike},  $\Phi\colon \cstu(X)\to\cstu(Y)$ is almost coarse-like, so there exists $s>0$ such that $\partial(y,y')\geq s$ implies \[|\langle\Phi(e_{x^1_nx^2_n})\delta_y,\delta_{y'}\rangle|<\delta^2\] for all $n\in\N$. Fix $n\in\N$ such that $\partial(y^1_n,y^2_n)> s$. 
 
Since $\Phi$ is rank preserving,  $\Phi(e_{xx})$ is a rank 1 projection for all $x\in X$. Hence, \cite[Lemma 6.5]{BragaFarahVignati2018} implies that
 \begin{align*}
 \delta^2&\leq \|\Phi(e_{x^2_nx^2_n})e_{y^2_ny^2_n}\|\cdot\|\Phi(e_{x^1_nx^1_n})e_{y^1_ny^1_n}\|\\
 &=\|e_{y^2_ny^2_n}\Phi(e_{x^1_nx^2_n})e_{y^1_ny^1_n}\|\\
 &=|\langle\Phi(e_{x^1_nx^2_n})\delta_{y^1_n},\delta_{y^2_n}\rangle|\\
 &<\delta^2;
 \end{align*}
 contradiction.
\end{proof}

\begin{theorem}\label{ThmRigidityUniformRoeAlgEmbeddings}
Let $X$ and $Y$ be u.l.f. metric spaces and assume that all the sparse subspaces of $Y$ yield only compact ghost projections. If   $\cstu(X)$ embeds into $\cstu(Y)$ by a rank  preserving map, then   there exists a uniformly finite-to-one  coarse map $X\to Y$.
\end{theorem}

\begin{proof}
Let $\Phi\colon \cstu(X)\to \cstu(Y)$ be a rank preserving strongly continuous embedding. Let $f\colon X\to Y$ be the uniformly finite-to-one map given by Lemma \ref{LemmaPickMapf}. By Lemma \ref{LemmaTheMapsAreCoarse}, $f$ is coarse.
\end{proof}

The following example shows that the conclusions of Theorem \ref{ThmRigidityUniformRoeAlgEmbeddingsFINITEUNION} and Theorem \ref{ThmRigidityUniformRoeAlgEmbeddings} are not equivalent.

\begin{proposition}\label{PropSpacesWhosePartitionEmbs}
There are u.l.f. metric spaces  $X$ and $Y$ such that $X$ cannot be mapped into $Y$ by a coarse uniformly finite-to-one map, but there exist a partition $X=X_1\sqcup X_2$ and
   injective coarse maps $X_i\to Y$ for $i\in \{1,2\}$.
\end{proposition}

\begin{proof} 
Let \[X_1=\N^2\cap\bigcup_{n\in\N}[n^2,n^2+n]\times[0,n]\text{ and }X_2=\N\times \{0\}\setminus X_1,\]
so $X_1,X_2\subset \N^2$ and $X_1\cap X_2=\emptyset$. Define \[X=X_1\sqcup X_2\text{ and }Y=X_1\times \{0\}\cup \{(0,0)\}\times \N,\]
so $X\subset \N^2$ and $Y\subset \N^3$, and endow $X$ and $Y$ with the standard subspace metrics (see Figure \ref{FigXY}), which we will denote by $d$. Clearly, $X_i$ can be mapped into $Y$ isometrically for $i\in\{1,2\}$.
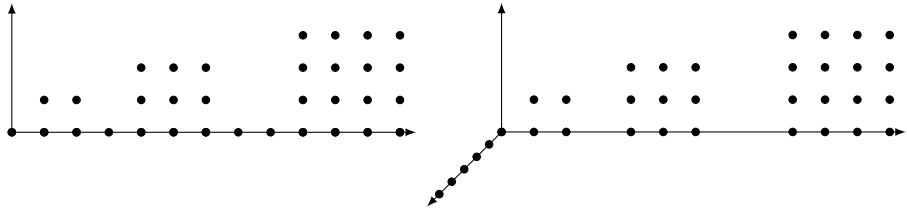
\begin{figure}[ht]
  \begin{tikzpicture}[scale=0.43]
    \coordinate (Origin)   at (0,0);
   \coordinate (XAxisMin) at (0,0);
    \coordinate (XAxisMax) at (12.5,0);
    \coordinate (YAxisMin) at (0,-2.3);
    \coordinate (YAxisMax) at (0,4);
\draw [thin, black,-latex] (XAxisMin) -- (XAxisMax);
    \draw [thin, black,-latex] (Origin) -- (YAxisMax);
     \draw [thin, white,-latex] (Origin) -- (YAxisMin);
  
     \foreach \x in {0}{
      \foreach \y in {0}{
        \node[draw,circle,inner sep=1pt,fill] at (\x,\y) {};
       }} 
     \foreach \x in {1,2}{
      \foreach \y in {0,1}{
        \node[draw,circle,inner sep=1pt,fill] at (\x,\y) {};
        }}
     \foreach \x in {4,5,6}{
      \foreach \y in {0,1,2}{
        \node[draw,circle,inner sep=1pt,fill] at (\x,\y) {};
}}
     \foreach \x in {9,10,11,12}{
      \foreach \y in {0,1,2,3}{
        \node[draw,circle,inner sep=1pt,fill] at (\x,\y) {};
  }}  
  
 \foreach \x in {0,1,2,3,4,5,6,7,8,9,10,11,12}{
      \foreach \y in {0}{
        \node[draw,circle,inner sep=1pt,fill] at (\x,\y) {};
  }}       
  \end{tikzpicture} \begin{tikzpicture}[scale=0.43]
    \coordinate (Origin)   at (0,0,0);
   \coordinate (XAxisMin) at (0,0,0);
    \coordinate (XAxisMax) at (12.5,0,0);
    \coordinate (YAxisMin) at (0,0,0);
    \coordinate (YAxisMax) at (0,4,0);
        \coordinate (ZAxisMin) at (0,0,0);
    \coordinate (ZAxisMax) at (0,0,6);

\draw [thin, black,-latex] (XAxisMin) -- (XAxisMax);
    \draw [thin, black,-latex] (YAxisMin) -- (YAxisMax);
     \draw [thin, black,-latex] (ZAxisMin) -- (ZAxisMax);
 
     \foreach \x in {0}{
      \foreach \y in {0}{
       \foreach \z in {0}{
        \node[draw,circle,inner sep=1pt,fill] at (\x,\y,\z) {};
       }} }
     \foreach \x in {1,2}{
      \foreach \y in {0,1}{
       \foreach \z in {0}{
        \node[draw,circle,inner sep=1pt,fill] at (\x,\y,\z) {};
        }}}
     \foreach \x in {4,5,6}{
      \foreach \y in {0,1,2}{
       \foreach \z in {0}{
 
        \node[draw,circle,inner sep=1pt,fill] at (\x,\y,\z) {};
}}}
     \foreach \x in {9,10,11,12}{
      \foreach \y in {0,1,2,3}{
       \foreach \z in {0}{
 
        \node[draw,circle,inner sep=1pt,fill] at (\x,\y,\z) {};
  }}  }
   \foreach \x in {0}{
      \foreach \y in {0}{
      \foreach \z in {0,1,2,3,4,5}{
        \node[draw,circle,inner sep=1pt,fill] at (\x,\y,\z) {};
  }}       }
  \end{tikzpicture}\caption{The metric spaces $X$ (left) and $Y$ (right)}\label{FigXY}
\end{figure}

Assume for a contradiction that there exists a coarse uniformly finite-to-one map $f\colon X\to Y$.  Define 
\[m_0=\sup\{d(f(x,y),f(z,w))\mid d((x,y),(z,w))\leq 1\},\]
so $m_0<\infty$. Since the $f\restriction \N\cup \{0\}$ is also coarse and uniformly finite-to-one,  there exists $k_0\in\N$ such that $f((n,0))\in \{(0,0)\}\times \N$ for all $n>k$. Without loss of generality, assume that 
\begin{enumerate}
\item\label{Item.1} $f((n,0))\in \{(0,0)\}\times (m_0,\infty)$,   for all $n>k_0$.
\end{enumerate}

Since $X_1$ (and its ``tails'') has asymptotic dimension $2$ and $\{(0,0)\}\times\N$ has asymptotic dimension $1$, the image of any ``tail'' of $X_1$ under $f$ must intersect $X_1\times \{0\}$. In other words, there exists a sequence $(x_i,y_i)_i$ in $X_1$ and a strictly increasing sequence of natural numbers $(n_i)_i$  such that
\begin{enumerate}\setcounter{enumi}{1}
\item\label{Item.2} $(x_i,y_i)\in [n_i^2,n_i^2+n_i]\times[0,n_i]$ for all $i\in\N$, and
\item\label{Item.3} $f((x_i,y_i))\in X_1\times \{0\}$ for all $i\in\N$.
\end{enumerate}
Notice that, using \eqref{Item.2}, for each $i\in\N$, there exists $\ell\in\N$ and a path $(x_i,y_i)=(x^{(0)}_i,y^{(0)}_i), \ldots,(x^{(\ell)}_i,y^{(\ell)}_i)=(n_i^2,0)$ such that $|(x^{(j)}_i,y^{(j)}_i)-(x^{(j+1)}_i,y^{(j+1)}_i)|=1$ for all $j\leq \ell-1$. Hence, since $f$ is coarse, \eqref{Item.1} and \eqref{Item.3} imply that
\[f^{-1}(\{(0,0)\}\times [0,m_0])\cap [n_i^2,n_i^2+n_i]\times[0,n_i]\neq \emptyset\]
for all $i\in\N$. This contradicts the fact that $f$ is uniformly finite-to-one.
\end{proof}

\subsection{Finite decomposition complexity (FDC)}\label{SubsectionFDC} 
This is a geometrical property formally stronger than property A\footnote{Notice that it is now known whether FDC and property A are equivalent.} and weaker than finite asymptotic dimension. 
We now show that it is  preserved by rank preserving embeddings of uniform Roe algebras. Since the definition of finite decomposition complexity is technical and long and we will not use it directly, we omit it and 
 refer the reader to either \cite[Definition 2.3]{GuentnerTesseraYu2012} or \cite[Definition 2.7.3]{NowakYuBook}. 

We need the following definitions bellow.

\begin{definition}
Let $\cX$ and $\cY$ be   families of metric spaces. 
\begin{enumerate}
\item The family $\cX$ is \emph{bounded} if $\sup_{X\in \cX}\diam(X)<\infty$.
\item The family $\cX$ is \emph{locally finite} if $\sup_{X\in \cX}\sup_{x\in X}|B_r(x)|<\infty$ for all $r> 0$.
\item The family $\cX$ is a \emph{subspace} of $\cY$ if every $X\in \cX$ is a subspace of some $Y\in \cY$.
\item A \emph{map of families} $F$ is a subset of $\bigcup_{X\in \cX,Y\in \cY}\{f\colon X\to Y\}$ such that for all $X\in \cX$ there exists $Y\in \cY$ and $f\colon X\to Y$ such that $f\in F$. We write $F\colon \cX\to \cY$.
\item A map of families $F\colon \cX\to \cY$ is \emph{uniformly finite-to-one} if there exists $k\in\N$ such that every $f\in F$ is $k$-to-one.
\end{enumerate}
\end{definition}

\begin{lemma}\label{LemmaRDecomposes}
Let $\cX$ be a family of metric spaces, $\cY$ be a family of u.l.f. metric spaces, and let $F\colon \cX\to \cY$ be a uniformly finite-to-one map of families. Let $\cV$ be a bounded subspace of $\cY$, and
\[  F^{-1}(\cV)=\{ f^{-1}(V)\mid V\in \cV, f\in F\}.\] Then for all $R>0$ there exists $L>0$ such that for all $U\in F^{-1}(\cV)$ there exists $k\in\N$ and $U_1,\ldots, U_k\subseteq U$ such that 
\begin{enumerate}
\item $U=\bigcup_{i=1}^kU_i$, 
\item $d(U_i,U_j)>R$ for all $i\neq j$, where $d$ is the metric of $U$, and
\item   $\diam(U_i)< L$ for all $i\leq k$.
\end{enumerate}
\end{lemma}

\begin{proof}
Fix $R>0$. Since $\cY$ is locally finite and $\cV$ is a bounded subspace of $\cY$, there exists $N\in\N$ such that $\sup_{V\in \cV}|V|<N$. Since $F$ is uniformly finite-to-one, it follows that  \[M=\sup_{U\in F^{-1}(\cV)}|U|<\infty.\]
The result now follows from Claim \ref{Claim1} for $L=MR$.
\end{proof}

\begin{corollary}\label{CorFDCFiniteToOne}
Let $X$ and $Y$ be u.l.f. metric spaces and assume that $X$ is mapped into $Y$ by a coarse uniformly finite-to-one map. If $Y$ has finite dimension complexity, then $X$ has finite dimensional complexity.
\end{corollary}

\begin{proofsketch}
This follows exactly as the proof of \cite[\S 3.1.3]{GuentnerTesseraYu2013}  (or \cite[Theorem 2.8.4]{NowakYuBook}) but with Lemma \ref{LemmaRDecomposes} substituting \cite[Lemma 3.1.2]{GuentnerTesseraYu2013} (\cite[Lemma 2.8.3]{NowakYuBook}).\qed
\end{proofsketch}

\begin{proof}[Proof of Corollary \ref{CorRigidityUniformRoeAlgEmbeddingsAsymDim}\eqref{CorRigidityUniformRoeAlgEmbeddingsAsymDim.2}]
If $Y$ has finite decomposition complexity, then $Y$ has property A \cite[Theorem 4.3]{GuentnerTesseraYu2013}. In particular, all sparse subspaces of $Y$ yield only compact ghost projections. By Theorem \ref{ThmRigidityUniformRoeAlgEmbeddingsFINITEUNION}, there exist a partition $X=\bigsqcup_{n=1}^kX_n$ such that $X_n$ is mapped by a uniformly finite-to-one coarse map into $Y$ for all $n\in\N$. So, \cite[3.1.3]{GuentnerTesseraYu2013} implies that $X_n$ has FDC for all $n\in \{1,\ldots,k\}$. By  \cite[3.1.7]{GuentnerTesseraYu2013}, $X$ has FDC.
\end{proof}

\subsection{The not compact preserving setting}
In this subsection, we show that the map $f\colon X\to Y$ obtained by Proposition \ref{LemmaPickMapf} can still be obtained without the compact preserving hypothesis on $\Phi\colon\cstu(X)\to \cstu(Y)$ if we assume that $Y$ satisfies a stronger geometric hypothesis, namely property A. We refer the reader to \S\ref{SectionOpenProb} for further  discussion on the not compact preserving case.

\begin{lemma}\label{LemmaPickMapfForPropertyA}
Let $X$ and $Y$ be  metric spaces and $\Phi\colon \cstu(X)\to\cstu(Y)$ be an embedding. Assume  $Y$ has property A. Then 
\[\delta=\inf_{x\in X}\sup_{y\in Y}\|\Phi(e_{xx})e_{yy}\|>0.\]
Moreover, if $f\colon X\to Y$ is a map such that $\|\Phi(e_{xx})e_{f(x)f(x)}\|\geq \delta$ for all $x\in X$, then $f$ is $\lceil \delta^{-2}\rceil$-to-one. 
\end{lemma}

\begin{proof}
By Proposition \ref{PropAlmostCoarseLikePropA}, $\Phi$ is almost coarse-like. Therefore, there exists $m\in\N$ such that $\Phi(e_{xx})$ can be $1/8$-$m$-approximated for all $x\in X$. For each $x\in X$, let $b_x\in \cstu(Y)$ be an element with propagation at most $m$ such that $\|\Phi(e_{xx})-b_x\|<1/4$. Since $\Phi$ is an embedding, $\|b_x\|> 3/4$ for all $x\in X$. Therefore, since $Y$ has property A, $Y$ has ONL and it follows that there exists $s>0$ such that for all $x\in X$, there exists a unit vector $\xi_x\in \ell_2(Y)$ with $\diam(\supp(\xi_x))<s$ such that $\|b_x\xi_x\|> 1/2$. So $\|\Phi(e_{xx})\xi_x\|>1/4$ for all $x\in X$. Since $Y$ is u.l.f., $M=\sup_{x\in X}|\supp(\xi_x)|<\infty$. Hence, for each $x\in X$, there exists $y\in \supp(\xi_x)$ such that $\|\Phi(e_{xx})\delta_y\|> 1/(4M)$. 

Given $\delta>0$, the fact that a map $f\colon X\to Y$ such that $\|\Phi(e_{xx})e_{f(x)f(x)}\|\geq \delta$ for all $x\in X$ is $\lceil \delta^{-2}\rceil$-to-one follows as in the proof of Lemma \ref{LemmaPickMapf}.
\end{proof}

Unfortunately, even though we are able of obtaining the map $f\colon X\to Y$ above, if one does not assume  the embedding $\Phi\colon \cstu(X)\to \cstu(Y)$ to be compact preserving, the map $f$ has no reason to be coarse of even for $X$ to have a finite partition $X=\bigsqcup_{n=1}^kX_n$ so that the restrictions $f\restriction X_n$ are coarse. 

\section{Embeddings onto hereditary subalgebras}\label{SectionEmbHere}

The main goal of this  section is to  prove Theorem \ref{ThmRigidityUniformRoeAlgEmbeddingsHereditary}. Although we have already seen that  embedding between uniform Roe algebras does not imply coarse embedding between the base spaces (Proposition \ref{PropositionEmbDoesNotImpCoarseEmb}), the situation changes if the image of the embedding is a hereditary subalgebra.

\begin{lemma}\label{LemmaPhiStronglyContAndU}
Let $(X,d)$ and $(Y,\partial)$ be  metric spaces and $\Phi\colon \cstu(X)\to\cstu(Y)$ be an embedding onto a hereditary \cstar-subalgebra of $\cstu(Y)$. Then 
\[
\Phi(\cK(\ell_2(X)))=\cK(\ell_2(Y))\cap \Phi(\cstu(X)).
\]
Moreover, there exists an  isometry $U\colon \ell_2(X)\to \ell_2(Y)$ such that $\Phi(a)=UaU^*$ for all $a\in \cstu(X)$. In particular,  $\Phi$ is strongly continuous and rank preserving. 
\end{lemma}

\begin{proof}
Let $A=\Phi(\cstu(X))$. Since $\cK(\ell_2(X))$ is an ideal in $\cstu(X)$, $\Phi(\cK(\ell_2(X)))$ is an ideal in $A$. Hence, as $A$ is hereditary, there exists an ideal $I$ of $\cstu(Y)$ such that $\Phi(\cK(\ell_2(X)))=I\cap A$. Since $I$ is an ideal in $\cstu(Y)$ and $\cstu(Y)$ contains all finite rank operators, $\cK(\ell_2(Y))\subseteq I$. So, \[\cK(\ell_2(Y))\cap A\subseteq \Phi(\cK(\ell_2(X))).\]
 Since $\cK(\ell_2(X)))$ is a simple \cstar-algebra, so is $\Phi(\cK(\ell_2(X)))$. Therefore, since $\cK(\ell_2(Y))\cap A$ is a nontrivial ideal of $\Phi(\cK(\ell_2(X)))$, we conclude that   
 \begin{equation}\label{EqHereCompPres}
 \cK(\ell_2(Y))\cap A= \Phi(\cK(\ell_2(X))).
 \end{equation}
 
Since $\Phi\restriction \cK(\ell_2(X))\colon \cK(\ell_2(X))\to \Phi(\cK(\ell_2(X)))$ is an isomorphism and $\Phi(\cK(\ell_2(X)))$ is a hereditary \cstar-subalgebra of $\cK(\ell_2(Y))$ (note if $\Phi$ is hereditary, it sends minimal projections to minimal projections),  it is standard and it follows, for example from adapting   \cite[Theorem 2.4.8]{MurphyBook}, that there exists  an (not necessarily surjective) isometry $U\colon \ell_2(X)\to \ell_2(Y)$ such that $\Phi(a)=UaU^*$ for all $a\in \cK(\ell_2(X))$. 

Let $H=U(\ell_2(X))$. Then $\Phi(a)H^\perp=0$ for all $a\in\cstu(X)$. Indeed, since $\Phi(a)=UaU^*$ for all $a\in \cK(\ell_2(X))$, it is clear that $\Phi(a)H^\perp=0$ for all $a\in \cK(\ell_2(X))$. Fix an arbitrary $a\in \cstu(X)$ and $\xi\in H^\perp$. Let $(b_n)_n$ be a sequence of compact operators in $\cstu(Y)$ converging to $\Phi(a)$ in the strong operator topology. Then $(\Phi(1)b_n\Phi(1))_n$ is a sequence of compact operators in $A$ which converges strongly to $\Phi(a)$. So, $\Phi(a)\xi=\lim_n\Phi(1)b_n\Phi(1)\xi$. By \eqref{EqHereCompPres}, $\Phi(1)b_n\Phi(1)\xi=0$ for all $n\in\N$, so $\Phi(a)\xi=0$.

Therefore, in order to show that  $\Phi(a)=UaU^*$ for all $a\in \cstu(X)$, it is enough to show that $\langle \Phi(a)e,f\rangle=\langle UaU^*e,f\rangle$ for all $e,f$ in an orthonormal basis of $H$.  
For that, let $a\in \cstu(X)$ and $x,y\in X$, and notice that
\begin{align*}
 \langle \Phi(a)U\delta_x,U\delta_y\rangle & =\langle e_{yy}U^*\Phi(a)Ue_{xx}\delta_x,\delta_y\rangle\\
&= \langle U^*Ue_{yy}U^*\Phi(a)Ue_{xx}U^*U\delta_x,\delta_y\rangle\\
 &=\langle U^*\Phi(e_{yy})\Phi(a)\Phi(e_{xx})U\delta_x,\delta_y\rangle\\
 &=\langle U^* \Phi(e_{yy}ae_{xx})U\delta_x,\delta_y\rangle\\
&=\langle e_{yy}ae_{xx}\delta_x,\delta_y\rangle\\
&=\langle UaU^*U\delta_x,U\delta_y\rangle
\end{align*}
(cf.  \cite[Lemma 3.1]{SpakulaWillett2013AdvMath}). 
 \end{proof}

\begin{lemma}\label{LemmaTheMapsAreExpanding}
Suppose $X$ and $Y$ are u.l.f. metric spaces and $\Phi\colon \cstu(X)\to \cstu(Y)$ is an embedding onto a hereditary \cstar-subalgebra of $\cstu(Y)$.  Then for all $r,\delta>0$  there exists $s>0$ such that for all $x_1,x_2\in X$ and all $y_1,y_2\in Y$, if  $d(x_1,x_2)\geq s$, $\|\Phi(e_{x_1x_1})e_{y_1y_1}\|\geq \delta$, and $\|\Phi(e_{x_2x_2})e_{y_2y_2}\|\geq \delta$, then $\partial (y_1,y_2)\geq r$.
\end{lemma}

\begin{proof}
Suppose otherwise. Then there exist 
 $r,\delta>0$, 
 sequences $(x^1_n)_n$ and $(x^2_n)_n$ in $X$, and sequences $(y^1_n)_n$ and $(y^2_n)_n$  in $Y$ such that $d(x_n^1,x_n^2)\geq n$, 
$\|\Phi(e_{x_n^1x_n^1})e_{y_n^1y_n^1}\|\geq \delta$,  $\|\Phi(e_{x_n^2x_n^2})e_{y_n^2y_n^2}\|\geq \delta$, and 
$\partial(y_n^1,y_n^2)\leq r$ for all $n\in\N$.

Since $d(x_n^1,x_n^2)\geq n$ for all $n\in\N$, by going to a subsequence, we can assume that either $(x^1_n)_n$ or $(x^2_n)_n$ are sequences of distinct elements. Without loss of generality, assume that  $(x^1_n)_n$ is a sequence of distinct elements.  

\begin{claim}
We can assume that both $(y^1_n)_n$ and $(y^2_n)_n$ are sequences of distinct elements.
\end{claim}

\begin{proof}
Suppose not. Then by going to a subsequence (and eventually swapping $y_n^1$ and $y_n^2$), we can assume that $(y^2_n)_n$ is constant. Then, since $Y$  is locally finite and $\partial (y^1_n,y^2_n)\leq r$ for all $n\in\N$, by going to a further subsequence, we can assume that $(y^1_n)_n$ is also constant. As $(x^1_n)_n$ is a sequence of distinct elements, $(\Phi(e_{x^1_nx^1_n}))_n$ is an orthogonal sequence of rank 1 projections. Hence, since $\|\Phi(e_{x^1_nx^1_n})e_{y^1_ny^1_n}\|\geq \delta$ for all $n\in\N$, this gives us a contradiction.
\end{proof}
Assume both $(y^1_n)_n$ and $(y^2_n)_n$ are sequences of distinct elements. Hence, since $(\Phi(e_{x^1_nx^1_n}))_n$ is an orthogonal sequence of   rank 1 projections, by going to a further subsequence, assume that 
\[\|e_{y^1_my^2_m}\Phi(e_{x^1_nx^1_n})\|<  2^{-n-m-1}{\delta^2}\]
for all $n\neq m$. 

Since $(y^1_n)_n$ and $(y^2_n)_n$ are sequences of distinct elements and $\partial(y^1_n,y^2_n)\leq r$ for all $n\in\N$, $\sum_{n\in\N}e_{y^1_ny^2_n}$ converges in the strong operator topology to an element in $\cstu(Y)$. As $\Phi(\cstu(Y))$ is hereditary, there exists $a\in \cstu(X)$ such that 
\[\Phi(a)=\Phi(1)\Big(\sum_{n\in\N}e_{y^1_ny^2_n}\Big)\Phi(1).\]

\begin{claim}\label{Claim1234}
$\inf_n\|e_{x^2_nx^2_n}ae_{x^1_nx^1_n}\|\geq \delta^2/2$.
\end{claim}

\begin{proof}
By \cite[Lemma 6.5]{BragaFarahVignati2018}, we have that
\[\|\Phi(e_{x^2_nx^2_n})e_{y^1_ny^2_n} \Phi(e_{x^1_nx^1_n})\|=
\|e_{y^2_ny^2_n}\Phi(e_{x^2_nx^2_n})\|\cdot\|e_{y^1_ny^1_n} \Phi(e_{x^1_nx^1_n})\|\geq \delta^2\]
for all $n\in\N$. Therefore, for all $n\in\N$, we have that 
\begin{align*}
\|e_{x^2_nx^2_n}ae_{x^1_nx^1_n}\|&=\Big\|\Phi(e_{x^2_nx^2_n})\Phi(a) \Phi(e_{x^1_nx^1_n})\Big\|\\
&=\Big\|\Phi(e_{x^2_nx^2_n})\Big(\sum_{m\in\N} e_{y^1_my^2_m}\Big) \Phi(e_{x^1_nx^1_n})\Big\|\\
&\geq \|\Phi(e_{x^2_nx^2_n})e_{y^1_ny^2_n} \Phi(e_{x^1_nx^1_n})\|- \sum_{m\neq n}\|\Phi(e_{x^2_nx^2_n})e_{y^1_my^2_m} \Phi(e_{x^1_nx^1_n})\| \\
&\geq \frac{\delta^2}{2},
\end{align*}
and the claim follows.
\end{proof}

Since $a\in \cstu(X)$ and $\lim_nd(x^1_n,x^2_n)=\infty$, Claim \ref{Claim1234} gives us a contradiction.
\end{proof}

\begin{proof}[Proof of Theorem \ref{ThmRigidityUniformRoeAlgEmbeddingsHereditary}(i)]
Let $\Phi\colon \cstu(X)\to \cstu(Y)$ be an isomorphism onto a hereditary \cstar-subalgebra. By Lemma \ref{LemmaPhiStronglyContAndU}, $\Phi$ is rank preserving. Let $f\colon X\to Y$ be the uniformly finite-to-one map given by Lemma \ref{LemmaPickMapf}. By Lemma \ref{LemmaTheMapsAreCoarse}, $f$ is coarse. By Lemma \ref{LemmaTheMapsAreExpanding}, $f$ is expanding, so we are done.
\end{proof}

\begin{proof}[Proof of Corollary \ref{CorRigidityUniformRoeAlgebra}]
By Theorem \ref{ThmRigidityUniformRoeAlgEmbeddingsHereditary}(i), it follows that $X$ coarsely embeds into $Y$. Since $Y$ coarsely embeds into a Hilbert space, so does $X$. The result now follows from \cite[Corollary 1.2]{BragaFarah2018}.
\end{proof}

\subsection{Property A and injectivity}
In this subsection, we show that if we assume $Y$ satisfies the stronger geometric condition of property A, then the coarse embedding given by Theorem \ref{ThmRigidityUniformRoeAlgEmbeddingsHereditary}(i) can be assumed to be injective (cf. \cite[Corollary 6.13]{WhiteWillett2017} and \cite[Theorem 1.11]{BragaFarahVignati2018}). We then prove Theorem~\ref{ThmRigidityUniformRoeAlgEmbeddingsHereditary}(ii).

\begin{assumption}\label{Assumption}
Throughout the remainder of this subsection, fix u.l.f. metric spaces $(X,d)$ and $(Y,\partial)$ with property A, and an embedding 
\[
\Phi\colon \cstu(X)\to \cstu(Y)
\]
 such that there exists an isometry $U\colon \ell_2(X)\to \ell_2(Y)$ so that $\Phi(a)=UaU^*$ for all $a\in \cstu(X)$.  Also, given   $x\in X$, $A\subseteq   X$,  
and $\delta>0$, we define
\[Y_{x,\delta}=\{y\in Y\mid \|\Phi(e_{xx})e_{yy}\|\geq \delta\}, \text{ and } Y_{A,\delta}=\cup_{x\in A}Y_{x,\delta}.\]
\end{assumption}

\begin{lemma}\label{LemmaUdelta}
For all $\gamma>0$ and all  $\varepsilon>0$ there exists $r>0$ such that 
for all   $A\subseteq X$ and all $B\subseteq   Y$  with $\|\Phi(\chi_A)\chi_B\Phi(1)\|>\gamma$, there exists $D\subseteq   X$ with $\diam(D)<r$ such that \[\|\Phi(\chi_{A\cap D})\chi_{B}\Phi(1)\|\geq (1-\varepsilon)\|\Phi(\chi_A)\chi_{B}\Phi(1)\|.\]
\end{lemma}

\begin{proof}
Fix $\gamma>0$ and $\varepsilon>0$. Define a map $\Psi\colon \cstu(Y)\to \cstu(X)$ by letting $\Psi(b)=U^*bU$ for all $b\in \cstu(Y)$. So, $\Psi$ is a compact preserving strongly continuous linear map, hence, by Proposition \ref{PropCoarseLike}, $\Psi$ is coarse-like. Since $X$ has property A, $X$ has the operator norm localization \cite[Theorem 4.1]{Sako2014}, so  \cite[Lemma 7.2]{BragaFarahVignati2018} gives $r>0$ such that for all $A\subseteq X$ and all $B\subseteq Y$ with $\|\chi_A\Psi(\chi_B)\|>\gamma$, there exists $D\subseteq X$ with $\diam(D)<r$ such that 
\begin{equation}\label{EqD}\|\chi_{A\cap D}\Psi(\chi_{B})\|\geq (1-\varepsilon)\|\chi_A\Psi(\chi_{B})\|.
\end{equation}
Fix  $A\subseteq X$ and  $B\subseteq Y$ with $\|\Phi(\chi_A)\chi_B\Phi(1)\|>\gamma$. Using that $U$ is an isometry onto the image of $\Phi(1)$, we have that
\begin{align*}
\|\chi_A\Psi(\chi_B)\|&=\|\chi_AU^*\chi_B\Phi(1)U\|\\&=\|U\chi_AU^*\chi_B\Phi(1)U\|\\
&=\|U\chi_AU^*\chi_B\Phi(1)\|\\
&=\|\Phi(\chi_A)\chi_B\Phi(1)\|>\gamma.
\end{align*} Let $D\subseteq X$ be as in \eqref{EqD} with $\diam(D)<r$. Analogously as in our previous computation, we have that 
\[\|\Phi(\chi_{A\cap D})\chi_{B}\Phi(1)\|=\|\chi_{A\cap D}\Psi(\chi_{B})\|\text{ and }\|\Phi(\chi_A)\chi_{B}\Phi(1)\|=\|\chi_A\Psi(\chi_{B})\|,\]
so we are done.\end{proof}

The following follows from either \cite[Lemma 6.7]{WhiteWillett2017} or \cite[Lemma 7.4(1)]{BragaFarahVignati2018}. Indeed, the proof of
\cite[Lemma 7.4(1)]{BragaFarahVignati2018} holds for any rank preserving $^*$-homomorphism $\Psi\colon \ell_\infty(X)\to \cstu(Y)$. Hence, we choose to omit its proof.

\begin{lemma}\label{LemmaONL}
For all $\varepsilon>0$ there exists $\delta>0$ such that  $\|\Phi(e_{xx})\chi_{Y_{x,\delta}}\|\geq 1-\varepsilon$, for all $x\in X$.
\qed
\end{lemma}

\begin{lemma}\label{LemmaCardinalityXBDelta}
For all $\eps>0$ there exists $\delta>0$ such that 
\[\|\Phi(\chi_A)(1-\chi_{Y_{A,\delta}})\Phi(1)\|<\eps,\]
for all finite subsets   $A\subseteq   X$. In particular, if $\eps>0$, then   $|A|\leq |Y_{A,\delta}|$ for all   $A\subseteq   X$. 
\end{lemma}

\begin{proof}
We prove the following stronger statement: for all $\eps>0$ there exists $\delta>0$ such that for all finite subsets   $A,B\subseteq   X$ with $A\subset B$,
\[\|\Phi(\chi_A)(1-\chi_{Y_{B,\delta}})\Phi(1)\|<\eps.\]
Suppose not, and fix $\eps>0$ and  sequences $(A_n)_n$ and $(B_n)_n$ of finite subsets of $X$ with, for all $n\in\N$, $A_n\subset B_n$ and $\|\Phi(\chi_{A_n})(1-\chi_{Y_{B_n,1/n}})\Phi(1)\|\geq 2\eps$.

\begin{claim}\label{C:7.7}
For every $n\in\N$ and every finite $F\subseteq   X $ there exist $m>n$,  a finite  $A\subseteq   X\setminus F$  and a finite $B\subset X$ with $A\subset B$ such that $\|\Phi(\chi_{A} )(1-\chi_{Y_{B,1/m}})\Phi(1)\|> \eps$.
\end{claim}

\begin{proof}
If not, fix  $n\in\N$ and $F\subseteq X$ such that  \[\|\Phi(\chi_{A} )(1-\chi_{Y_{B,1/m}})\Phi(1)\|\leq \eps\] for all finite $A\subseteq   X\setminus F$, all $B\subset X$ with $A\subset B$ and all $m>n$. Since $F$ is finite, pick $m>n$ large enough so that $\|\Phi(\chi_{C} )(1-\chi_{Y_{C,1/m}})\|< \eps$ for all $C\subseteq   F$. In particular,  $\|\Phi(\chi_{C} )(1-\chi_{Y_{D,1/m}})\|< \eps$ for all $C\subseteq   F$ and all $D\subset X$ with $C\subset D$. The contradiction comes from
\begin{align*}
\|\Phi(\chi_{A_m})(1-\chi_{Y_{B_m,1/m}})\Phi(1)\|&\leq \|\Phi(\chi_{A_m\cap F})(1-\chi_{Y_{B_m,1/m}})\Phi(1)\|\\
&\ \ \ \ +\|\Phi(\chi_{A_m\setminus F})(1-\chi_{Y_{B_m,1/m}})\Phi(1)\|< 2\eps.\qedhere
\end{align*}
\end{proof}

By Claim~\ref{C:7.7}, redefining the sequence $(A_n)_n$ and passing to a subsequence, we assume that the $A_n$'s are finite and disjoint, and that 
 \[
 \|\Phi(\chi_{A_n})(1-\chi_{Y_{B_n,1/n}})\Phi(1)\|> \eps,
 \] 
 for all $n\in\N$. 
 
 Let $W_n=Y\setminus Y_{B_n,\delta}$. Since $(A_n)_n$ is a disjoint sequence of finite subsets and $\|\Phi(\chi_{A_n})\chi_{W_n}\Phi(1)\|>\eps$ for all $n$, \cite[Lemma 3.4]{BragaFarahVignati2018} allows us to pick a sequence $(Y_n)_n$ of disjoint finite subsets of $Y$ such that $\|\Phi(\chi_{A_n})\chi_{W_n\cap	 Y_n}\Phi(1)\|>\eps/2$ for all $n\in\N$.  Therefore,  Lemma~\ref{LemmaUdelta} implies that there exists $r>0$ and a sequence $(D_n)_n$ of subsets of $X$ such that  $\diam(D_n)<r$ and 
\begin{equation}\label{EqONL1}
\|\Phi(\chi_{A_n\cap D_n})\chi_{W_n\cap Y_n}\Phi(1)\|\geq (1-\eps)\|\Phi(\chi_{A_n})\chi_{W_n\cap Y_n}\Phi(1)\|,\tag{$*$}
\end{equation}
for all $n\in\N$.

Since $X$ is u.l.f. and $\sup_n\diam (D_n)\leq r$, there exists $N\in\N$ so that $\sup_n|D_n|<N$. Pick $\theta>0$ small enough so that $4N\theta^{1/2} <\eps(1-\eps) $.
By Lemma \ref{LemmaONL}, pick $n\in \N$ large enough so that $\|\Phi(e_{xx})\chi_{Y_{x,1/n}}\|\geq 1-\theta$ for all $x\in X$. Since, for all $x\in X$, $\Phi(e_{xx})$ is a rank 1 projection, then $\Phi(e_{xx})\chi_{Y_{x,1/n}}\Phi(e_{xx})=\lambda\Phi(e_{xx})$ for some $\lambda\geq (1-\theta)^2$, and therefore  $\|\Phi(e_{xx})(1-\chi_{Y_{x,1/n}})\|<2\theta^{1/2}$ for all $x\in X$. It follows that 

\begin{align*}\label{EqONL3}
\|\Phi(\chi_{A_n\cap D_n })\chi_{W_n\cap Y_n}\Phi(1)\|&\leq\|\Phi(\chi_{A_n\cap D_n })\chi_{W_n}\|\tag{$**$}\\ 
&\leq \sum_{x\in A_n\cap D_n}\|\Phi(e_{xx})(1-\chi_{Y_{x,1/n}})\|\\
&\leq 4\theta^{1/2}|D_n|\\
&\leq \frac{\eps(1-\eps)}{2},
\end{align*}
for all $n\in\N$. Therefore, inequalities \eqref{EqONL1} and \eqref{EqONL3} imply that
  \[\|\Phi(\chi_{A_n})\chi_{B_n\cap Y_n}\Phi(1)\|\leq \frac{\eps}{2}\] for all $n\in\N$; contradiction.

We are left to show that, if $\eps>0$ then our choice of $\delta$ implies that $|A|\leq |Y_{A,\delta}|$ for all $A\subseteq  X$. Fix $\eps>0$ and let $\delta>0$  be given by the first statement of the lemma. Notice that $|A|=\rank \Phi(\chi_A)$ and $|Y_{A,\delta}|=\rank( \chi_{Y_{A,\delta}})$. Suppose $\rank( \Phi(\chi_A))>\rank( \chi_{Y_{A,\delta}})$.  Let $H=U(\ell_2(X))$, so $\Phi(1)$ is the projection onto $H$.  Since $\mathrm{corank}(1-\chi_{Y_{A,\delta}})=\rank \chi_{Y_{A,\delta}} $, we can pick
\[\xi\in \mathrm{Im}(1-\chi_{Y_{A,\delta}})\cap\mathrm{Im}(\Phi(\chi_A))\]
with norm 1. Since $\xi\in \mathrm{Im}(\Phi(\chi_A))\subseteq H$, this gives us that   $\|\Phi(\chi_A)( 1-\chi_{Y_{A,\delta}})\Phi(1)\xi\|=\|\xi\|=1$; contradiction.
\end{proof}

\begin{lemma}\label{LemmaExistenceInjection}
There exists $\delta>0$ and  an injection  $f\colon X\to Y$ such that  $f(y)\in Y_{x,\delta}$ for all $x\in X$. 
\end{lemma}

\begin{proof}
Let $\delta$ be given by Lemma \ref{LemmaCardinalityXBDelta} for some $\varepsilon>0$. Define a map $\alpha\colon X\to \cP(Y)$ by letting $\alpha(x)=Y_{x,\delta}$ for all $x\in X$. Since $Y_{A,\delta}=\bigcup_{x\in A}\alpha(x)$, the choice of $\delta$ gives that
\[|A|\leq |Y_{A,\delta}|=\Big|\bigcup_{x\in A}\alpha(x)\Big|.\]
Therefore, by Hall's marriage theorem, the required injection exists.
\end{proof}

\begin{proof}[Proof of Theorem \ref{ThmRigidityUniformRoeAlgEmbeddingsHereditary}(ii)]
We are going to prove that, in $Y$ has property A, the fact that $\cstu(X)$ is isomorphic to a hereditary \cstar-subalgebra of $\cstu(Y)$ is equivalent to the existence of a coarse injection $X\to Y$. 

For the forward direction, let $\Phi\colon \cstu(X)\to \cstu(Y)$ be an isomorphism onto a hereditary \cstar-subalgebra and assume that $Y$ has property A. By Theorem \ref{ThmRigidityUniformRoeAlgEmbeddingsHereditary}(i), $X$ coarsely embeds into $Y$, so $X$ has property A. By Lemma \ref{LemmaPhiStronglyContAndU}, there exists an isometry $U\colon \ell_2(X)\to \ell_2(Y)$ such that  $\Phi(a)=UaU^*$ for all $a\in \cstu(X)$. Hence, Assumption \ref{Assumption} holds.   Let $\delta>0$ and $f\colon X\to Y$ be the injective map given by \ref{LemmaExistenceInjection}. By Lemma \ref{LemmaTheMapsAreCoarse}, $f$ is coarse. By Lemma \ref{LemmaTheMapsAreExpanding}, $f$ is expanding, so we are done.

For the backward direction, if $f\colon X\to Y$ is an injective coarse embedding, then $f\colon X\to f(X)$ is a bijective coarse equivalence, so $\cstu(X)$ and $\cstu(f(X))$ are isomorphic (e.g., \cite[Theorem 8.1]{BragaFarah2018}). Since $\cstu(f(X))=\chi_{f(X)}\cstu(f(X))\chi_{f(X)}$, it follows that $\cstu(f(X))$ is a hereditary \cstar-subalgebra of $\cstu(Y)$, so we are done.
\end{proof}

\section{Digression on our geometric property}\label{SectionGeomProp}

Given a u.l.f. metric space $X$, the property of all of $X$'s sparse subspaces yielding only compact ghost projections is central in these notes. Therefore, it is only natural to try to understand this property better. The most natural question, in the context of large scale geometry, is whether this property is preserved under coarse embeddings/equivalences. Although we were not able to answer this question, there is a similar property which  is indeed a coarse invariant (Theorem \ref{ThmMetricPropertyIsCoarse}). For that, we need to introduce a new object: the   \emph{stable Roe algebra of $X$}. See \cite{SpakulaWillett2013AdvMath} for more details on stable Roe algebras and other types  of Roe algebras (see also Remark \ref{RemarkStableRoe} below).

Let $X$ be a metric space and $H$ be a Hilbert space.  We identify $\ell_2(X)\otimes H$ with $\ell_2(X,H)$ via the isomorphism which sends $\delta_x\otimes v$ to the map $u\colon X\to H$ such that $u(x)=v$ and $u(y)=0$ for all $y\neq x$. Given $x\in X$, let $e_{xx}$ denote the operator on $\ell_2(X,H)$ such that 
\[e_{xx}\delta_z\otimes v=\langle\delta_z,\delta_x\rangle\delta_x\otimes v.\]
For each $a\in \cB(\ell_2(X,H))$ and $x,y\in X$, let $a_{xy}=e_{yy}ae_{xx}$. So $a_{xy}$ can be canonically identified with an element in $\cB(H)$.

\begin{definition}\label{DefiStableRoe}
Let $X$ be a u.l.f. metric space and $H$ be the infinite dimensional separable Hilbert space. The \emph{stable Roe algebra of $X$}, denoted by $\mathrm{C}^*_s(X)$, is the closure in $\cB(\ell_2(X,H))$ of the algebra of all operators $a\in \cB(\ell_2(X,H))$ which satisfy the following:
\begin{enumerate}
\item $a$ has  \emph{finite propagation}, i.e.,  there exists $r>0$ such that $a_{xy}=0$ for all $x,y\in X$ with $d(x,y)>r$, and
\item\label{ItemStableRoe} there exists a finite dimensional subspace $H_a\subset H$ such that $a_{xy}\in \cB(H_a)$ for all $x,y\in X$.
\end{enumerate}
\end{definition}

Clearly, the uniform Roe algebra coincides with $\csts(X)$ if $H=\mathbb C$. Notice also that $\csts(X)$ contains $\cK(\ell_2(X,H))$ and it is canonically isomorphic to $\cstu(X)\otimes \cK(H )$. We now define ghost operators in $\csts(X)$.

\begin{definition}\label{DefiGhostStable}
Let $X$ be a metric space and let $a\in \csts(X)$. We say that $a$ is a \emph{ghost} if for all $\varepsilon>0$ there exists a bounded $ A\subseteq X$  such that $\|a_{x,y}\|<\varepsilon$ for all $x,y\in X\setminus A$.
\end{definition}

We can now define the desired geometric property which we prove it is a coarse invariant.

\begin{definition}
Let $X$ be a metric space. We say that the \emph{all sparse subspaces of $X$ yield only compact ghost projections in $\csts(X)$} if given a sparse subspace $X'\subseteq X$, all the ghost projections in $\csts(X')$ are compact. 
\end{definition}

Clearly, the  property above is (at least formally) more restrictive, than  the property of all sparse subspaces of $X$ yielding only compact ghost projections (in $\cstu(X)$). On the other hand, if $X$ coarsely embeds into a Hilbert space, then all sparse subspaces of $X$ yield only compact ghost projections in $\csts(X)$ -- this follows analogously as  \cite[Lemma 7.3]{BragaFarah2018}.

\begin{lemma}\label{LemmaBlockNonCompGhostProj}
Let $X$ be a metric space and assume that $X'=\bigsqcup_nX_n$ is a sparse subspace of $X$ so that $\csts(X')$ has a noncompact ghost projection. Then there exists a noncompact ghost projection $p\in \csts(X')$ such that $\chi_{X_n}p\chi_{X_m}=0$ for all $n\neq m$.
\end{lemma}

\begin{proof}
Let $q\in \csts(X')$ be a noncompact ghost projection. Clearly,  it is enough to find a projection $p\in\csts(X')$ with   $\chi_{X_n}p\chi_{p_m}=0$ for all $n\neq m$ and such that $p-q$ is compact. For each $n\in\N$ let $q_n=\chi_{X_n}q\chi_{X_n}$ and 
\[
q'=\mathrm{SOT}\text{-}\lim_N\sum_{n=1}^Nq_n\in\prod_n\cB(\ell_2(X_n,H)).
\]
\begin{claim}\label{ClaimBlockNonCompGhostProj1}
The operator $q-q'$ is compact. In particular, $q'\in \csts(X')$. 
\end{claim}

\begin{proof}
Let $\varepsilon>0$. Since $q\in \csts(X')$, there exists $a\in \csts(X')$ of finite propagation such that $\|q-a\|<\varepsilon$. Since $X'=\bigsqcup_nX_n$ is sparse, we can write $a=c+w$, where both $c$ and $w$ have finite propagation, $c\in \csts(X')\cap \prod_n\cB(\ell_2(X_n,H))$, $w\in \cK(\ell_2(X',H))$, and $\chi_{X_n}w\chi_{X_n}=0$ for all $n\in\N$. Therefore, 
\begin{align*}
\|q'-c\|&=\sup_{n}\|q_n-\chi_{X_n}c\chi_{X_n}\|\\
&=\sup_{n}\|\chi_{X_n}q\chi_{X_n}-\chi_{X_n}(c+w)\chi_{X_n}\|\\
&\leq \|q-(c+w)\|\\
&<\varepsilon.
\end{align*}
Hence, 
\[\|q-q'-w\|\leq \|q-(c+w)\|+\|c-q'\|\leq 2\varepsilon.\]
Since $w$ is compact and $\varepsilon>0$ was arbitrary, we conclude that  $q-q'$ is compact. In particular,  $q'=q-(q-q')\in \csts(X')$. 
\end{proof}
Let $A_n=\cB(\ell_2(X_n,H))$. Consider the canonical quotient map
\[
\pi\colon \csts(X)\to\csts(X)/\mathcal K(\ell_2(X,H)).
\]
Then 
\[
\pi(q)\in\pi\Big(\prod A_n\Big)=\prod A_n/\Big(\prod A_n\cap\mathcal K(\ell_2(X,H))\Big)\cong \prod_nA_n/\bigoplus_n A_n
\]
and $\pi(q)$ is a projection. By a standard argument, we can lift a projection in $\prod_nA_n/\bigoplus_n A_n$ to a projection in $\prod_n A_n$, so there is projection $p\in\prod_n\cB(\ell_2(X_n,H)$ such that $\pi(p)=\pi(q)$.
\end{proof}

\begin{theorem}\label{ThmMetricPropertyIsCoarse}
Let $X$ be a metric space, $Y$ be  u.l.f. metric space and assume that there exists a uniformly finite-to-one coarse map $f\colon X\to Y$. If the sparse subspaces of $Y$ yield only compact ghost projections in $\csts(Y)$, then the same holds for $X$. 
\end{theorem}

\begin{proof}
First, we are going to show that we can assume that $f$ is injective. Let $\partial$ be the distance on $Y$. Pick $k\geq 1$ so that $f$ is $k$-to-one. Let $Y_k=Y\times \{1,\ldots,k\}$ and 
define a distance $\partial_k$ on $Y_k$ as 
\[
\partial_k((y,i),(y',j))=\begin{cases}d_Y(y,y'),&\text{ if }y\neq y'\\
1,&\text{ if }y=y'.
\end{cases}
\] 
Note that 
\[
\csts(Y_k)\cong\csts(Y)\otimes M_k\cong\cstu(Y)\otimes\mathcal K(H)\otimes M_k\cong\csts(Y),
\]
where $M_k$ denotes the \cstar-algebra of $k\times k$-matrices. Moreover, if a sparse subspace $Y'\subseteq Y_k$ yields a noncompact ghost projection in $\csts(Y_k)$, then the corresponding sparse subspace $\{y\in Y\mid\exists i\leq k ((y,i)\in Y')\}$ yields a noncompact ghost projection in $\csts(Y)$. In particular, sparse subspaces of $Y$ yield noncompact ghosts projections in $\csts(Y)$ if and only if the same happens to sparse subspaces of $Y_k$. For each $y\in Y$, enumerate $f^{-1}(y)=\{x^y_1,\ldots,x^y_i\}$ for some $i\leq k$. Then each $x\in X$ is equal to 
$x^y_j$ for a unique $y$ and $j$,  and we can define $f_k\colon X\to Y_k$  as $f_k(x^y_j)=(y,j)$. Then $f_k$ is coarse and injective. In order to avoid heavy notation, we let $Y=Y_k$, $\partial=\partial_k$ and $f=f_k$.

We now proceed by contradiction. Suppose that there is a sparse subspace  $X'=\bigsqcup X_n\subseteq X$  yielding a noncompact ghost projection $p\in\csts(X)$. By Lemma \ref{LemmaBlockNonCompGhostProj}, we can assume that $p\in\prod_n\mathcal B(\ell_2(X_n,H))$. Let $Z_n=f'(X_n)$. Although $\bigsqcup_n Z_n$ might not be a sparse subspace of $Y$ because $f$ is not expanding, there is an infinite $M\subseteq \N$ such that $Z'=\bigsqcup_{n\in M}Z_n$ is sparse. Let $q=\chi_{\bigcup_{n\in M} X_n}$, and $p'=qpq$. Then $p'$ is a noncompact ghost projection. Define $r\in\cB(\ell_2(Y,H))$ by 
\[
r_{yy'}=\begin{cases}
p'_{xx'},&\text{ if }y=f(x), y'=f(x')\\
0,&\text{ otherwise.}
\end{cases}
\] 
Since $f$ is coarse, $r=(r_{yy'})_{yy'\in Y}$ belongs to $\csts(Y)$. Moreover $r$ is a noncompact ghost projection such that $r\in\prod_{n\in M}\cB(\ell_2(Y_n,H))$. Therefore 
 $Z'$ is a sparse subspace of $Y$ yielding a noncompact ghost projection. This is a contradiction.\end{proof}

\begin{corollary}
The property for a u.l.f. metric space $X$ of its sparse subspaces  yielding  only compact ghost projections in $\csts(X)$ is a coarse property.
\end{corollary}

\begin{proof}
Notice that, since $X$ is u.l.f.,  any coarse embedding $X\to Y$  is uniformly finite-to-one. So the result follows from Theorem \ref{ThmMetricPropertyIsCoarse}.
\end{proof}

\begin{remark}\label{RemarkStableRoe}
Similarly to Definition \ref{DefiGhostStable}, we can define ghost operators in the \emph{Roe algebra $\mathrm{C}^*(X)$} and in the \emph{uniform algebra $\mathrm{UC}^*(X)$} (see \cite[Example 2.2]{SpakulaWillett2013AdvMath} for definitions). The proofs of Lemma \ref{LemmaBlockNonCompGhostProj} and Theorem \ref{ThmMetricPropertyIsCoarse} remains valid for  the property of all sparse subspaces yielding only compact ghost projections in $\mathrm{C}^*(X)$ (or $\mathrm{UC}^*(X)$).
\end{remark}

\section{Open problems}\label{SectionOpenProb}

In this subsection, we gather the most important questions which this paper leaves open. Firstly, our methods do not take care of not compact preserving embeddings $\cstu(X)\to \cstu(Y)$. It is therefore not clear whether the conclusion of Corollary \ref{CorRigidityUniformRoeAlgEmbeddingsAsymDim} holds for arbitrary embeddings.

\begin{problem}\label{ProblemEmbPreservesAsymDim}
Let $X$ and $Y$ be u.l.f.  metric spaces such that $\cstu(X)$ embeds into $\cstu(Y)$.  
\begin{enumerate}
\item Does it follow that $\mathrm{asydim}(X)\leq \mathrm{asydim}(Y)$?
\item If $Y$ has FDC, does $X$ have FDC as well?
\end{enumerate}
\end{problem}

A counterexample to Problem~\ref{ProblemEmbPreservesAsymDim} would have to be very different in nature from the spaces constructed in Proposition~\ref{PropositionEmbDoesNotImpCoarseEmb}. A version of either Theorem \ref{ThmRigidityUniformRoeAlgEmbeddingsFINITEUNION} or Theorem \ref{ThmRigidityUniformRoeAlgEmbeddings} for not compact preserving embeddings would give us a positive answer to Problem \ref{ProblemEmbPreservesAsymDim}.

\begin{problem}\label{P.1}
Let $X$ and $Y$ be u.l.f. metric spaces and assume that $\cstu(X)$ embeds into $\cstu(Y)$. Does it follow that there exists a uniformly finite-to-one coarse map $X\to Y$, or at least that 
$X$ has a partition $X=\bigsqcup_{n=1}^kX_n$ such that each $X_k$ can be mapped injectively into $Y$ by  a coarse map? What if $Y$ has property A, or if it coarsely embeds into a Hilbert space? 
\end{problem}

Although  Proposition \ref{PropSpacesWhosePartitionEmbs} implies that the conclusion of Theorem~\ref{ThmRigidityUniformRoeAlgEmbeddingsFINITEUNION} is strictly weaker than the existence of a uniformly finite-to-one coarse map $X\to Y$, the following local variant of Problem~\ref{P.1} remains open.

\begin{problem}
Are there  u.l.f. metric spaces $X$ and $Y$ such that $\cstu(X)$ embeds into $\cstu(Y)$ by a compact preserving map but so that there is no uniformly  finite-to-one  coarse map $X\to Y$? 
\end{problem}

By Theorem \ref{ThmEmbImpliesStrongEmb}, we can often without  loss of generality assume that embeddings are strongly continuous. It would be very interesting to obtain an analogous result for rank  preservation 
and compact preservation. In particular, this would imply a positive answer to Problem \ref{ProblemEmbPreservesAsymDim}. Precisely, we ask the following.

\begin{problem}\label{ProblemObtainRankPreservation}
Let  $X$ and $Y$ be  u.l.f. metric spaces,  and  $\Phi\colon \cstu(X)\to \cstu(Y)$ be an embedding. Is there a projection $p\in \cstu(Y)$  in $\Phi(\cstu(X))'\cap \cstu(Y)$ such that  
 \[
 a\in \cstu(X)\mapsto p\Phi(a)p\in \cstu(Y)
 \] 
is a rank preserving embedding? What if it is compact preserving? What if $Y$ has property A?
\end{problem}

Corollary \ref{CorRigidityUniformRoeAlgEmbeddingsAsymDim} gives us geometric properties which are preserved under compact preserving embeddings of uniform Roe algebras. It would be interesting to add property A and coarse embeddability into a Hilbert space to this list of properties. Because of that, we ask the following.

\begin{problem}\label{ProblemPropAEmbHilb}
Let $X$ and $Y$ be u.l.f. metric spaces and assume that there exists a uniformly finite-to-one coarse map $X\to Y$. 
\begin{enumerate}
\item\label{ProblemPropAEmbHilb.1} If $Y$ has property A, does it follow that $X$ has property A?
\item If $Y$ coarsely embeds into a Hilbert space, does the same hold for $X$?
\end{enumerate}
\end{problem}

In  \cite[Theorem 1.1]{Sako2012} it has been announced that a u.l.f. metric space has property A if and only if $\cstu(X)$ is exact. Since a \cstar-subalgebra of an exact algebra is also exact, this result would give a positive answer to Problem \ref{ProblemPropAEmbHilb}\eqref{ProblemPropAEmbHilb.1}. 
The special case when $X$ is a group with the Cayley metric is well-known (\cite[Theorem 5.1.6]{BrownOzawa}).

At last, in the spirit of \S\ref{SectionGeomProp}, it would be interesting to know whether our main geometric property is a coarse invariant.

\begin{problem}
Let $X$ and $Y$ be u.l.f. metric spaces and assume that $Y$ coarsely embeds into $X$. If all sparse subspaces of $X$ yield only compact ghost projections, does the same hold for $Y$? What if $X$ and $Y$ are coarsely equivalent? 
\end{problem}

\end{document}